\documentclass[11pt,leqno]{amsproc}

\usepackage{anysize}
\marginsize{3.5cm}{3.5cm}{2.5cm}{2.5cm}
\usepackage[]{hyperref}
\hypersetup{
    colorlinks=true,       % false: boxed links; true: colored links
    linkcolor=red,          % color of internal links
    citecolor=blue,        % color of links to bibliography
    filecolor=magenta,      % color of file links
    urlcolor=cyan           % color of external links
}

\usepackage{amsmath}
\usepackage{amsfonts,amssymb}
\usepackage{enumerate}
\usepackage{mathrsfs}
\usepackage{amsthm}
\usepackage{bbm}
\usepackage{csquotes}

\theoremstyle{plain}
\newtheorem{theo}{Theorem}[section]
\newtheorem{prop}[theo]{Proposition}
\newtheorem{lemm}[theo]{Lemma}
\newtheorem{coro}[theo]{Corollary}

\newtheorem{defi}[theo]{Definition}
\theoremstyle{definition}
\newtheorem{rema}[theo]{Remark}
\newtheorem{nota}[theo]{Notation}
\DeclareMathOperator{\cnx}{div}

\DeclareMathOperator{\Op}{Op}

\DeclareMathOperator{\supp}{supp}

\DeclareSymbolFont{pletters}{OT1}{cmr}{m}{sl}
\DeclareMathSymbol{s}{\mathalpha}{pletters}{`s}

\def\lDx#1{\langle D_x\rangle^{#1}\,}

\def\B{B }

\def\Cr{\mathcal{C}}

\def\defn{\mathrel{:=}}

\def\eps{\varepsilon}

\def\la{\left\vert}
\def\lA{\left\Vert}
\def\le{\leq}
\def\les{\lesssim}
\def\leo{}

\def\mez{\frac{1}{2}}

\def\ra{\right\vert}
\def\rA{\right\Vert}

\def\tdm{\frac{3}{2}}

\def\xC{\mathbf{C}}
\def\xN{\mathbf{N}}
\def\xR{\mathbf{R}}

\def\xZ{\mathbf{Z}}
\def\cF{ \mathcal{F}}

\def\lb{\left[}
\def\lB{\left\{}
\def\lp{\left(}

\def\rb{\right]}
\def\rB{\right\}}
\def\rp{\right)}

\def\tph{{\widetilde\varphi}}
\def\ld{\lambda}
\def\gs{\gtrsim}
\def\th{\widetilde{h}}
\newcommand{\bq}{\begin{equation}}
\newcommand{\eq}{\end{equation}}
\newcommand{\bqa}{\begin{eqnarray*}}
\newcommand{\eqa}{\end{eqnarray*}}

\newcommand{\hk}{\hspace*{.15in}}
\numberwithin{equation}{section}

\title{Sharp Strichartz estimates for water waves systems}
\author{
Huy Quang Nguyen
\address{
Huy Quang Nguyen. Program in Applied $\&$ Computational Mathematics, Princeton University, Princeton, NJ 08544.
}
\email{qn@math.princeton.edu}
}
\begin{document}
%\date{\empty}
\begin{abstract}
Water waves are well-known to be dispersive at the linearization level. Considering the fully nonlinear systems,  we prove for reasonably smooth solutions the optimal  Strichartz estimates for pure gravity waves and the semi-classical Strichartz estimates for gravity-capillary waves; for both 2D and 3D waves. Here, by optimal we mean the gains of regularity (over the Sobolev embedding from Sobolev spaces to H\"older spaces) obtained for the linearized systems. Our proofs combine the paradifferential  reductions of Alazard-Burq-Zuily \cite{ABZ1, ABZ3} with a dispersive estimate using a localized wave package type parametrix of Koch-Tataru \cite{KT}.
\end{abstract}
\thanks{The author was supported in part by Agence Nationale de la Recherche
  project  ANA\'E ANR-13-BS01-0010-03.}
\maketitle
\section{Introduction}
\hk Water waves systems govern the dynamic of an  interface between a fluid domain and the vacuum. It is well-known that these systems are dispersive, {\it i.e.}, waves at different frequency propagate at different speed. For approximate models of water waves in certain regimes  such as Kadomtsev-Petviashvili  equations, Korteweg-de Vries  equations, Schrodinger equations, wave equations, {\it etc}, dispersive properties have been extensively studied. For the fully nonlinear system of water waves, dispersive properties are however less understood. \\
\hk On the one hand, in  global dynamic, dispersive properties have been considered in establishing the existence of global (or almost global) solutions for small, localized, smooth  data by the works of Wu \cite{Wu09, Wu11}, Germain-Masmoudi-Shatah \cite{GMS, GMS1}, Ionescu-Pusateri \cite{IoPu, IoPu1}, Alazard-Delort \cite{AD}, Ifrim-Tataru \cite{IfTa, IfTa2}. On the other hand, in local dynamic, dispersive properties and more precisely Strichartz estimates have been exploited in proving the existence of local-in-time solutions with rough, generic data,  initiated by the work of Alazard-Burq-Zuily \cite{ABZ4} and then followed by de Poyferr\'e-Nguyen \cite{NgPo1, NgPo2}, Nguyen \cite{Ng}. Prior to these, a Strichartz estimate was proved for 2D gravity-capillary waves by  Christianson-Hur-Staffilani \cite{CHS}. Unlike the case of semilinear Schrodinger (wave) equations, water waves systems are quasilinear in nature and thus how much regularity one can gain in Strichartz estimates depends also on the smoothness of solutions under consideration. In other words, in term of dispersive analysis (for generic solutions), the nonlinear systems are not obviously dictated by their linearizations. In fact, the Strichartz estimates proved in \cite{ABZ4, NgPo2, Ng} are non optimal compared to the linearized systems. We address in this paper the following problem:\\
{\it At which level of regularity, solutions to the fully nonlinear systems of water waves obey the same Strichartz estimates as their linearizations?}\\
Remark first that due to the systematic use of symbolic calculus in the framework of semi-classical analysis in \cite{ABZ4, NgPo2, Ng} we were not able to reach sharp Strichartz estimates by simply adapting the method there to the case of sufficiently smooth solutions.\\
In this paper, we choose to work on the Zakharov-Craig-Sulem formulation of water waves, which is recalled now.
\subsection{The Zakharov-Craig-Sulem formulation of water waves}
We consider an incompressible inviscid fluid with unit density moving in a time-dependent domain  
$$
\Omega = \{(t,x,y) \in[0,T] \times \xR^d \times \xR:(x, y)\in \Omega_t\} 
$$
where each $\Omega_t$ is a domain located underneath a free surface 
$$
\Sigma_t = \{(x,y)  \in \xR^d \times \xR: y=\eta(t, x)\} 
$$
and above a fixed bottom $\Gamma=\partial\Omega_t\setminus \Sigma_t$. We make the following assumption on the domain: {\it Assumption $(H)$\\
Each $\Omega_t$ is the intersection of the half space 
\[
\Omega_{1,t}= \{(x,y)  \in\xR^d \times \xR: y<\eta(t, x)\} 
\]
and an open connected set $\Omega_2$ containing a fixed strip around $\Sigma_t$, {\it i.e.}, there exists $h>0$ such that  for all $t\in [0, T]$
\[
 \{(x,y)  \in \xR^d \times \xR: \eta(x)-h\le y\le\eta(t, x)\} \subset \Omega_2.
\]
}
 Assume that the velocity field $v$ admits a potential $\phi:\Omega \to \xR$, {\it i.e.}, $v=\nabla \phi$. Using the Zakharov formulation, we introduce the trace of $\phi$ on the free surface
$$\psi(t,x)= \phi(t,x,\eta(t,x)).$$ 
 Then $\phi(t, x, y)$ is the unique variational solution of 
\bq\label{phi}
\Delta\phi =0\text{~in}~\Omega_t,\quad \phi(t, x, \eta(t, x))=\psi(t, x).
\eq
The Dirichlet-Neumann operator is then defined by
\begin{align*}
G(\eta) \psi &= \sqrt{1 + \vert \nabla_x \eta \vert ^2}
\Big( \frac{\partial \phi}{\partial n} \Big \arrowvert_{\Sigma}\Big)\\
&= (\partial_y \phi)(t,x,\eta(t,x)) - \nabla_x \eta(t,x) \cdot(\nabla_x \phi)(t,x,\eta(t,x)).
\end{align*}
The {\it gravity-capillary waves} (see \cite{LannesLivre}) problem  consists in solving the following system of  $(\eta,\psi)$:
\begin{equation}\label{ww:0}
\left\{
\begin{aligned}
&\partial_t \eta = G(\eta) \psi,\\
&\partial_t \psi + g\eta+\sigma H(\eta)+\mez \vert \nabla_x \psi \vert^2 - \mez \frac{(\nabla_x \eta \cdot \nabla_x \psi + G(\eta)\psi)^2}{1+ \vert \nabla_x \eta \vert^2}=0,
\end{aligned}
\right.
\end{equation}
where $\sigma$ is the surface tension coefficient and $H(\eta)$ is the mean curvature of the free surface:
\[
H(\eta)=-\cnx\left( \frac{\nabla\eta}{\sqrt{1+|\nabla\eta|^2}}\right).
\]
In the regime of large wavelengths, one can discard the effect of surface tension by taking $\sigma=0$ in the system \eqref{ww} to obtain the system of {\it pure gravity water waves}
\begin{equation}\label{ww}
\left\{
\begin{aligned}
&\partial_t \eta = G(\eta) \psi,\\
&\partial_t \psi + g\eta+\mez \vert \nabla_x \psi \vert^2 - \mez \frac{(\nabla_x \eta \cdot \nabla_x \psi + G(\eta)\psi)^2}{1+ \vert \nabla_x \eta \vert^2}=0,
\end{aligned}
\right.
\end{equation}
The physical dimensions are $d=1, 2$. For terminologies, when $d=1$ (respectively $d=2$) we call \eqref{ww:0}, \eqref{ww} the 2D (respectively 3D)  waves systems. It is important to introduce the vertical and horizontal components of the trace of the velocity on $\Sigma$, which can be expressed in terms of $\eta$ and $\psi$:
\begin{equation}\label{BV}
B = (v_y)\arrowvert_\Sigma = \frac{ \nabla_x \eta \cdot \nabla_x \psi + G(\eta)\psi} {1+ \vert \nabla_x \eta \vert^2},\quad V= (v_x)\arrowvert_\Sigma  =\nabla_x \psi - B \nabla_x \eta.
 \end{equation}
 We recall also that the Taylor coefficient defined by $a = -\frac{\partial P}{\partial y}\big\arrowvert_\Sigma$ can be defined in terms of $\eta,\psi,B,V$ only (see \S 4.2 in \cite{ABZ3}).
\subsection{Known results and main theorems}
\subsubsection{Pure gravity water waves}
For the system \eqref{ww} of pure gravity water waves, the only existent Strichartz estimate, to our knowledge, is \cite{ABZ4} where the authors proved Strichartz estimates for rough solutions with a gain of 
\bq\label{gain:ABZ4}
\mu=\frac{1}{24}-~\text{when}~d=1,\quad \mu=\frac{1}{12}-~\text{when}~d\ge 2
\eq
The starting point of this result is the symmetrization of \eqref{ww:0} into a quasilinear paradifferential equation of the following form (see Appendix \ref{appendix} for the paradifferential calculus theory and Theorem \ref{theo:sym} below for a precise reduction statement)
\bq\label{eq:sym:0}
\lp\partial_t +T_V\cdot\nabla +iT_\gamma \rp u=f\in L^\infty_t H^s_x,~s>1+\frac{d}{2},
\eq
where $\gamma$ is a symbol of order $\mez$. \\
Let us now look at the linearization of \eqref{ww} (take $g=1$ and  infinite depth) around the rest state $(0, 0)$:
\[
\begin{cases}
\partial_t\eta-|D_x|\psi=0,\\
\partial_t\psi+\eta=0,
\end{cases}
\]
which is equivalent to, after imposing $u:=\eta+i|D_x|^\mez\psi $,
\bq\label{eq:lin}
\partial_t u+i|D_x|^\mez u=0.
\eq
For this Schrodinger-type dispersive equation we can prove classically the Strichartz estimates
\bq\label{Str:lin}
\lA u\rA_{L^pW^{s-\frac{d}{2}+\mu_{opt}, \infty}}\le C(s,d) \lA u\arrowvert_{t=0}\rA_{H^s},\quad \begin{cases}
\mu_{opt}=\frac{1}{8},~p=4\quad\text{if}~d=1, \\
\mu_{opt}=\frac{1}{4}-,~p=2\quad\text{if}~d\ge 2,
\end{cases}
\eq
from which the estimates for the original unknowns $\eta, \psi$ can be recovered. Our first result states that the fully nonlinear system \eqref{ww} enjoys  Strichartz estimates with the same gain as in \eqref{Str:lin}, for solutions slightly smoother than the energy threshold in \cite{ABZ3}.
\begin{nota}
Denote \begin{align*}
&\mathcal{H}^s=H^{s+\mez}(\xR^d)\times H^{s+\mez}(\xR^d)\times H^s(\xR^d)\times H^s(\xR^d),\\
&\mathcal{W}^s=W^{r+\mez,\infty}(\xR^d)\times W^{r+\mez,\infty}(\xR^d)\times W^{r,\infty}(\xR^d)\times W^{r,\infty}(\xR^d).
\end{align*}
\end{nota}
\begin{theo}\label{main:theo}
Let $d=1,~ 2$ and consider a solution $(\eta, \psi)$ of \eqref{ww} on the time interval $I=[0, T],~T<+\infty$ such that $\Omega_t$ satisfies $H_t$ for every $t\in [0, T]$ and 
\[
(\eta, \psi, B, V)\in C([0, T]; \mathcal{H}^s).
\]
(see Theorem $1.2$, \cite{ABZ3}). Define 
\[
\begin{cases}
s(d)=\frac{5}{3}+\frac{d}{2},~\mu_{opt}(d)=\frac{1}{8},~p(d)=4\quad\text{if}~d=1, \\
s(d)=2+\frac{d}{2},~\mu_{opt}(d)=\frac{1}{4}-,~p(d)=2\quad\text{if}~d\ge 2.
\end{cases}
\]
Then for any $s>s(d)$ we have
\[
(\eta, \psi, B, V)\in L^{p(d)}(I;\mathcal{W}^{s-\frac{d}{2}+\mu_{opt}(d),\infty}).
\]
\end{theo}
\subsubsection{Gravity-capillary waves} Let us now look at the linearization of \eqref{ww:0} (with infinite depth) around the rest state $(0, 0)$,
\[
\begin{cases}
\partial_t\eta-|D_x|\psi=0,\\
\partial_t\psi-\Delta \eta=0
\end{cases}
\]
or equivalently, with  $\Phi=|D_x|^\mez\eta+i\psi$
\bq\label{ww:lin}
\partial_t\Phi +i|D_x|^\tdm \Phi=0,
\eq
for which one can easily prove the following Strichartz estimates
\bq\label{Str:lin}
\lA \Phi\rA_{L^pW^{s-\frac{d}{2}+\mu_{opt}, \infty}}\le C(s, d)\lA \Phi\arrowvert_{t=0}\rA_{H^s},\quad \begin{cases}
\mu_{opt}=\frac{3}{8},~p=4\quad\text{if}~d=1, \\
\mu_{opt}=\frac{3}{4}-,~p=2\quad\text{if}~d\ge 2.
\end{cases}
\eq
Turning to the nonlinear case, in high dimensions ($d\ge 2$)  the geometry can be non trivial and hence trapping can occur. As a consequence, natural dispersive estimates expected are the one constructed at small time scales which are tailored to the frequencies. The propagator $e^{-it|D_x|^{\tdm}}$ has the speed of propagation of order $|\xi|^{\mez}$. Hence, for time $|t|<|\xi|^{-\mez}$, we do not expect to encounter any problem due to the global geometry. This leads to the so called {\it semi-classical Strichartz estimate}. This  terminology appeared in \cite{BGT} for a study of the Schr\"odinger equations on compact manifolds. To realize this heuristic argument, one multiplies both sides of \eqref{ww:lin} by $h^\tdm$ with $h=2^{-j},~j\in \xN$ and make a change of temporal variables $t=h^\mez \sigma,~u(\sigma, x)=\Phi(h^\mez\sigma, x)$ to derive the semi-classical equation
\bq\label{ww:lins}
h\partial_\sigma u+|hD_x|^\tdm u=0.
\eq
Then the optimal dispersive estimates for \eqref{ww:lins} implies the semi-classical Strichartz estimates for \eqref{ww:lin} with a lost of $\frac{1}{8}$ derivatives when $d=1$ and $\frac{1}{4}$ derivatives when $d\ge 2$.\\
In \cite{ABZ1} it was proved that if
\bq\label{reg}
(\eta, \psi)\in C([0, T]; H^{s+\mez}\times H^s), \quad s>2+\frac{d}{2},
\eq
then system \eqref{ww:0} can be symmetrized into a single equation analogous to its linearization \eqref{ww:lin}:
\bq\label{ww:sym}
\lp\partial_t +T_V\cdot\nabla +iT_\gamma\rp u=f\in L^\infty H^s,\quad \gamma\in \Gamma^\tdm,
\eq
from which the local-wellposedness was obtain at this regularity level-\eqref{reg}.
Using this reduction, Alazard-Burq-Zuily  \cite{ABZ2} established, for 2D waves, the semi-classical Strichartz estimate at the threshold \eqref{reg} and  the classical (optimal) Strichartz estimate when $s>5+\mez$. We remark that in \cite{ABZ1}, the semi-classical gain is achieved due to the fact that after a {\it para change of variables} (see Proposition $3.4$, \cite{ABZ2}), the highest order term $T_\gamma u$ in \eqref{ww:sym} is converted into the simple Fourier multiplier $|D_x|^\tdm$. Unfortunately, such a reduction can not work for the 3D case and hence, the semi-classical Strichartz estimate in this case is much more difficult to establish, especially at the regularity level \eqref{reg}. In the present paper, we aim to investigate the semi-classical Strichartz estimate for \eqref{ww:0} when $d \ge 2$, assuming that the solution is slightly smoother than \eqref{reg} ($1/2$ derivatives). Our second result reads as follows. 
\begin{theo}\label{main:theo:0}
Let $d\ge 2$ and $0<T<\infty$.  Consider a solution $(\eta, \psi)$ of \eqref{ww:0} on the time interval $I=[0, T]$ such that $\Omega_t$ satisfies $H(t)$ for every $t\in [0, T]$ and
\[
(\eta, \psi)\in C([0, T]; H^{s+\mez}(\xR^d)\times H^s(\xR^d)).
\]
If $s>\frac{5}{2}+\frac{d}{2}$ then for every $\eps>0$, there holds
\[
(\eta, \psi)\in L^{2}([0, T]; W^{s+1-\eps-\frac{d}{2}}(\xR^d)\times W^{s+\mez-\eps-\frac{d}{2},\infty}(\xR^d)).
\]
\end{theo}
\begin{rema}
Our proof of Theorem \ref{main:theo:0} works equally for the 2D waves ($d=1$), when $(\eta, \psi)\in C([0, T]; H^{s+\mez}(\xR)\times H^s(\xR))$ with $s>\frac{5}{2}+\mez$. On the other hand, using the paracomposition reduction of Proposition $3.4$ in \cite{ABZ2} we can indeed improve the preceding regularity to $s>2+\mez$, which is the same as Theorem $1.1$ in \cite{ABZ2}. 
\end{rema}
\subsubsection{On the proof of the main results}
 In \cite{ABZ4, NgPo2} the authors worked completely in the semi-classical formalism and proved dispersive estimates using the approximation WKB method.  This allowed the authors to prove Strichartz estimates with nontrivial gains even for very rough backgrounds. However, we emphasize that with this method, we were not able to reach the classical or semi-classical level as in Theorems \ref{main:theo} and \ref{main:theo:0}. The dispersive estimates for principally normal pseudo-differential operators in \cite{KT} require more regularity ($C^2$) of the symbols to control the Hamiltonian flow and apply the FBI transform technique. This  allows us to obtain sharp dispersive estimates when the characteristic set of the symbol has maximal nonvanishing principal curvatures. \\
For the proof of our main results, we shall combine the paradifferential reduction in the works of Alazard-Burq-Zuily  with the phase transform method in the work \cite{KT} of Koch-Tataru. Notice that the later works effectively for operators of order $1$, after renormalizing. For gravity-capillary waves (see \eqref{ww:sym}) the dispersive term has order $\frac{3}{2}$ and thus the semi-classical time-scale brings it to the one of order $1$ and hence leads to the semi-classical Strichartz estimate in Theorem \ref{main:theo:0}. For the pure gravity waves \eqref{eq:sym:0}, one observes that the dispersive term $iT_\gamma$ has order $\mez$ which is lower than that of the transport term $T_V\cdot\nabla$. Here, we follow \cite{ABZ4}, suppressing this transport term by straightening the vector field $\partial_t+T_V\cdot\nabla$ and then make another change of spatial variables to convert it to an operator of order $1$. However, the new symbol then is not in the standard form $p(x, \xi)$ to apply phase transforms and other technical issues appear. Thus, the proof of Theorem \ref{main:theo} requires much more care.
%%%%%%%%%%%%%%%%%%%%%%%%%%%%%%%%%
\section{Preliminaries}
\subsection{Symmetrization of system \eqref{ww}}
We first recall the symmetrization of system \eqref{ww} to a single quasilinear equation, performed in \cite{ABZ1}. This reduction requires the following symbols:
\begin{itemize}
\item Symbols of the Dirichlet-Neumann operator
\begin{gather*}
\lambda^{(1)}:=\sqrt{(1+|\nabla \eta|^2)|\xi|^2-(\nabla \eta\cdot\xi)^2},\\
\lambda^{(0)}:=\frac{1+|\nabla \eta|^2}{2\lambda^{(1)}}\left\{ \cnx( \alpha^{(1)}\nabla\eta)+i\nabla_\xi\lambda^{(1)}\cdot\nabla_x\alpha^{(1)}\right\},~~\alpha^{(1)}:=\frac{\lambda^{(1)}+i\nabla\eta\cdot\xi}{1+|\nabla\eta|^2}.
\end{gather*}
\item Symbols of the mean-curvature operator:
\[
{\ell}^{(2)}:=(1+|\nabla \eta|^2)^{-\frac{1}{2}}\left(|\xi|^2-\frac{(\nabla \eta\cdot\xi)^2}{1+|\nabla\eta|^2}\right), \quad {\ell}^{(1)}:=-\frac{i}{2}(\partial_x\cdot\partial_{\xi})\ell^{(2)};
\]
\item Symbols using for symmetrization
\[
q:=\left(1+(\nabla_x\eta)^2 \right)^{-\frac{1}{2}},~p=\left(1+(\nabla_x\eta)^2 \right)^{-\frac{5}{4}}|\xi|^\mez+p^{(-\mez)} ,
\] 
where $p^{(-\mez)}:=F(\nabla_x\eta, \xi)\partial^\alpha_x\eta$,  with $|\alpha|=2$ and 
$F\in C^\infty(\xR^d\times \xR^d\setminus\{0\}; \xC)$ homogeneous of order $-\mez$ in $\xi$.
\item Symbols in the symmetrized equation:
\begin{align*}
&\gamma:=\sqrt{l^{(2)}\lambda^{(1)}}=\left(\frac{\la\xi\ra^2(1+\la\nabla\eta\ra^2)-(\nabla\eta\cdot\xi)^2}{1+\la\nabla\eta\ra^2}\right)^\frac{3}{4}, \\
& \omega:=-
      \frac{i}{2}(\partial_\xi\cdot\partial_x)\sqrt{l^{(2)}\lambda^{(1)}},\quad \omega_1:=\sqrt{\frac{l^{(2)}}{\lambda^{(1)}}}\frac{\Re\lambda^{(0)}}{2}.
\end{align*}
\end{itemize}
Define  the {\it good-unknown} $U:=\psi-T_B\eta$. It was proved in \cite{ABZ1} the following result.
\begin{theo}[\protect{\cite[Corollary 4.9]{ABZ1}}]\label{theo:sym:0} Let  $s>2+\frac{d}{2}$  and $(\eta, \psi)\in C^0([0, T]; H^{s+\mez}\times H^s)$ be a solution to  \eqref{ww:0} and satisfies condition $(H_t)$ for every $t\in [0, T]$. The complex-valued function $u:=T_p\eta+iT_qU$ then solves the following paradifferential equation
	\begin{equation}\label{ww:reduce:0}
		\partial_tu+ T_V\cdot\nabla  u+iT_{\gamma+\omega+\omega_1} u=f,
	\end{equation}
	where there exists a nondecreasing function $\cF:\xR^+\times \xR^+\to \xR^+$ in dependent of $(\eta, \psi)$ such that
	\begin{equation}\label{est:RHS}
		\lA f\rA_{L^\infty([0, T]; H^s)}\leq\cF\left(\lA(\eta, \psi)\rA_{L^\infty([0, T]; H^{s+\mez}\times H^s)}\right).
	\end{equation}
\end{theo}
\subsection{Symmetrization of system \eqref{ww:0}}
Define first the principle symbol of the Dirichlet-Neumann operator
$$
\ld=\Bigl(\big(1+ \vert \nabla \eta\vert^2\big)\vert \xi \vert^2 
- \big(\xi \cdot \nabla \eta\big)^2\Bigr)^\mez.
$$ 
Next, set $\zeta=\nabla\eta$ and introduce
$$
U_s \defn \lDx{s} V+T_\zeta \lDx{s}\B,\quad \zeta_s\defn \lDx{s}\zeta.
$$
\begin{theo}[\protect{\cite[Proposition 4.10]{ABZ3}}]\label{theo:sym} Let  $s>1+\frac{d}{2}$  and $(\eta, \psi)\in C^0([0, T]; H^{s+\mez}\times H^s)$ be a solution to  \eqref{ww} such that  condition $(H_t)$ is fulfilled for every $t\in [0, T]$ and the velocity trace
\[
(B, V)\in C^0([0, T]; H^{s+\mez}\times H^s),
\]
and there exists $c_0>0$ such that $a(t, x)\ge c_0$ for all $(t,x)\in [0, T]\times \xR^d$.  Then the complex-valued function $$u:=\lDx{-s}(U_s-iT_{\sqrt{a/\lambda}}\zeta_s)$$%\in C^0([0, T]; H^s(\xR^d)$$
solves the following paradifferential equation
	\begin{equation}\label{ww:reduce}
		\partial_tu+ T_V\cdot\nabla  u+iT_\gamma u=f,
	\end{equation}
	where $\gamma=\sqrt{a\ld}$ and 
\[
\lA a-g\rA_{L^\infty([0, T]; H^{s-\mez})}+\lA f\rA_{L^\infty([0, T]; H^s)}\le \cF\left( \lA (\eta,\psi)\rA_{H^{s+\mez}},\lA (V,B)\rA_{H^s}\right).
\]
\end{theo}
\begin{rema}\label{inverse}
The change of variables $(\eta, \psi)\mapsto u$ in Theorem \ref{theo:sym:0} and $(\eta, \psi, B, V)\mapsto u$ in Theorem \ref{theo:sym} are essentially ``invertible" in the sense that one can recover Sobolev estimates and H\"older estimates for $(\eta, \psi, B, V)$  from those for $u$ by virtue of the symbolic calculus for paradifferential operators contained in Theorem \ref{theo:sc}.
\end{rema}
\subsection{Para and pseudo differential operators}
Since the paradifferential setting is not suitable for proving dispersive estimates, we shall change it into the pseudo-differential setting, whose standard definitions are recalled here.
\begin{defi}
1. For any $m\in \xR,~0\le \delta_1, \delta_2, \rho\le 1$ we denote by $S^m_{\rho,\delta_1, \delta_2}$ the class of all symbols $a(x, y, \xi):(\xR^d)^3\to \xC$ satisfying
\[
\la \partial_x^\alpha\partial_y^\beta\partial_\xi^\gamma a(x, y, \xi)\ra \le C_{\alpha, \beta, \gamma}(1+|\xi|)^{m+\delta_1|\alpha|+\delta_2|\beta|-\rho|\gamma|}.
\]
The corresponding pseudo-differential operator is defined by 
\[
\Op(a)u(x)=\int_{\xR^d}e^{i(x-y)\xi}a(x, y, \xi)u(y)dy d\xi.
\]
When $a:(\xR^d)^2\to \xC$ we consider it as a symbol in $S^m_{\rho, \delta_1, 0}$ that does not depend on $y$ and rename $S^m_{\rho, \delta_1, 0}\equiv S^m_{\rho,\delta_1}$.\\
2. For any symbol $a(x, \xi)\in S^m_{\rho, \delta}$ the Weyl quantization $\Op^w(a)\equiv a^w(x, D_x)$ is defined by $\Op^w(a)u(x)=\Op(b)u(x)$
with $b(x, y, \xi):=a(\frac{x+y}{2}, \xi)\in S^m_{\rho, \delta, \delta}$.
\end{defi}
We shall later need to transform the operators $\Op(a)$ to $\Op^w(a)$. This is done by means of the following result, which can be easily deduce from  \cite{Taylor}, Proposition 0.3.A.
\begin{prop}\label{expand}
For any symbol $a\in S^m_{\rho, \delta}$ with $m\in \xR,~0\le \delta<\rho\le 1$ there exists a symbol $b\in S^m_{\rho, \delta}$ such that  $\Op^w(a)=\Op(b)$. Moreover, we have the following asymptotic expansion in the sense of symbolic calculus:
\[
b(x, \xi)\sim \sum_{|\alpha|\ge 0}\frac{(-i)^{|\alpha|}}{\alpha!2^{|\alpha|}}\partial_x^\alpha\partial_\xi^\alpha a (x, \xi).
\]
Remark that for all $\alpha \in \xN^d$, $\partial_x^\alpha\partial_\xi^\alpha a (x, \xi)\in S_{\rho, \delta}^{m-(\rho-\delta)|\alpha|}$.
\end{prop}
 Now, let $a\in \Gamma^m_r,~r>0$ be a paradifferential symbol (see Definition \ref{defi:para}) and define
\bq\label{smoothout}
\forall j\in \xZ,~\forall \delta>0, \quad S_{j\delta}(a)(x,\xi)=\psi(2^{-j\delta}D_x)a(x,\xi)
\eq
 the spatial regularization of the symbol~$a$, where~$\psi$ is the Littlewood-Paley function defined in~\eqref{defi:psi}. We first prove a Bernstein's type inequality for $S_{j\delta}(a)$. 
\begin{lemm}\label{Bernstein}
If $a\in \Gamma^m_\rho$  then for all  $\alpha, \beta\in \xN^d,~|\alpha|\ge \rho$, there exists a constant $C_{\alpha, \beta}$ such that for all $(x, \xi)\in \xR^{2d}$
\[
\vert \partial_x^\alpha\partial_\xi^\beta S_{j\delta}(a)(x, \xi)\vert\le C_{\alpha, \beta}2^{j\delta(|\alpha|-\rho)}\Vert \partial_\xi^\beta a(\cdot, \xi)\Vert_{W^{\rho, \infty}(\xR^d)}.
\]
\end{lemm}
\begin{proof}
If $|\alpha|=\rho$ the estimate is obvious by writing $\partial_x^\alpha\partial_\xi^\beta S_{j\delta}(a)$ as a convolution of $\partial_x^\alpha\partial_\xi^\beta a$ with a kernel. Considering now $|\alpha|>\rho$. Recall the dyadic partition of unity \eqref{partition}: $1=\sum_{k=0}^\infty \Delta_k$ where each $\Delta_k$ is spectrally supported in the annulus $\{2^{k-1}\le |\xi|\le 2^{k+1}\}$. Using this partition, we can write
$$
\partial_x^\alpha \partial^\beta_\xi  S_{j\delta}(a)(x,\xi)
= \sum_{k =0}^{+\infty}\Delta_k \partial_x^\alpha \psi(2^{-j\delta} D_x) \partial^\beta_\xi a(x,\xi):= \sum_{k=0}^{+\infty}u_k
$$
If $\mez 2^{k} \geq  2^{j\delta}$ then $\Delta_k \psi(2^{-j\delta} D_x) =0$. 
Therefore
$$
\partial_x^\alpha \partial^\beta_\xi  S_{j\delta}(a)(x,\xi)= \sum_{k=0}^{2+ [j\delta]} u_k.
$$
Now, introducing $\varphi_1(\xi)\in C^\infty_c(\xR^d)$, supported in $\{\frac{1}{3} \leq \xi \vert \leq 3\}$ one has   
$$
u_k= 2^{k \vert \alpha \vert}  \varphi_1(2^{-k}D_x) \psi(2^{-j\delta} D_x) \Delta_k \partial^\beta_\xi a(x,\xi).
 $$
Consequently,
$$
\Vert u_k\Vert_{L^\infty(\xR^d)} 
\leq 2^{k\vert \alpha \vert} \Vert  {\Delta}_k  D^\beta_\xi  a(\cdot,\xi)\Vert_{L^\infty(\xR^d)} 
\leq C2^{k\vert \alpha \vert} 2^{- k\rho} 
\Vert \partial^\beta_\xi a(\cdot, \xi)\Vert_{W^{\rho, \infty}(\xR^d)}.
$$
It follows that
$$
\Vert \partial_x^\alpha \partial^\beta_\xi  S_{j\delta}(a)(x,\xi)\Vert _{L^\infty(\xR^d)} 
\leq C  \sum_{k =0}^{2+ [j\delta]} 2^{k(\vert \alpha \vert - \rho)} \Vert D^\beta_\xi  a (\cdot, \xi) \Vert_{W^{\rho, \infty}(\xR^d)}.
$$
Finally, since $\vert \alpha\vert - \rho > 0$ we deduce that
$$\Vert \partial_x^\alpha \partial^\beta_\xi  S_{j\delta}(a)(x,\xi)\Vert _{L^\infty(\xR^d)} \le  C  2^{j\delta(\vert \alpha \vert - \rho)}\Vert \partial^\beta_\xi a (\cdot, \xi) \Vert_{W^{\rho, \infty}(\xR^d)},
$$
which concludes the proof.
\end{proof}
We show in the next Proposition that after localizing a distribution $u$ in frequency, one can go from paradifferential operators to pseudo-differential operators when acting on $u$. 
\begin{prop}\label{para:pseudo}
 For every $j\in \xN^*$, define 
$$ R_ju:=T_a\Delta_ju-S_{j-3}(a)(x, D_x)\Delta_ju.$$
Then $R_ju$ is spectrally supported (see Definition \ref{spectral}) in  an annulus $\{c_1^{-1}2^{j-1}\le |\xi|\le c_12^{j+1}\}$ and for every $\mu\in \xR$ we have
$$\lA R_ju\rA_{H^{\mu-m+r}(\xR^d)}\le CM^m_r(a)
			      \lA u\rA_{H^\mu(\xR^d)}
$$
where the constants $c_1,~C>0$ are independent of $a, u, j$.
\end{prop}
\begin{proof} Recall first the definition \eqref{eq.para} of $T_au$, where we have $\varrho=1$ on the support of $\varphi_j$ for any $j\ge 1$, so
$$ R_ju=T_a\Delta_ju-S_{j-3}(a)(x, D_x)\varrho(D_x)\Delta_ju.$$
In the following proof, we shall use the presentation of M\'etivier \cite{MePise} on pseudo-differential and paradifferential operators. To be compatible with \cite{MePise} we also abuse  notations: by $\Gamma^m_r$ we denote the class of symbols $a$ satisfying the growth condition \eqref{para:symbol} for any $\xi\in \xR^d$ and by $M_0^m$ the semi-norm \eqref{defi:semi-norm} where the suppremum is taken over $\xi\in\xR^d$.\\
 \hk 1. By definition \eqref{eq.para} it holds that $T_a v=\Op(\sigma_a\varrho)v$,
where $\Op(\sigma_a\varrho)$ denotes the classical pseudo-differential operator with symbol
\[
\sigma_a(x, \xi)\varrho(\xi)=\chi(D_x, \xi)a(x,\xi)\varrho(\xi).
\]
Hence $R_ju=\Op(a_j)u$ with
\[
a_j(x, \xi)=\sigma_a(x, \xi)\varrho(\xi)\varphi_j(\xi)-S_{j-3}(a)(x, \xi)\varrho(\xi)\varphi_j(\xi).
\]
 Now, we write 
\[
a_j=\big(\sigma_a\varrho\varphi_j-a\varrho\varphi_j\big)+\big (a\varrho\varphi_j- S_{j-3}(a)\varrho\varphi_j\big)=a_j^1+a_j^2.
\]
Applying Proposition $5.8 (ii)$ in \cite{MePise} gives $a_j^1\in \Gamma^{m-r}_0$ and (remark that $(\varphi_j)_j$ is bounded in $\Gamma^0_r$)
\[
M_0^{m-r}(a_j^1)\le CM_r^m(a\varrho\varphi_j)\le CM^m_r(a\rho).
\]
On the other hand, if we denote  $b=a\varrho\varphi$ then $a_j^2(x, \xi)=b(x,\xi)-\psi_{j-3}(D_x, \xi)b(x, \xi)$. Taking into account the fact that $\supp \varphi_j\subset B(0, C2^j)$  we may estimate
\begin{align*}
|a_j^2(x, \xi)|&\le \sum_{k\ge j-2}|\Delta_jb(x, \xi)|\le \sum_{k\ge j-2}
2^{-kr}\lA b(\cdot, \xi)\rA_{W^{r, \infty}}\\
&\le C2^{-jr}\lA b(\cdot, \xi)\rA_{W^{r, \infty}}=C2^{-jr}|\varphi_j(\xi)|\lA a(\cdot, \xi)\varrho(\xi)\rA_{W^{r, \infty}}\\
&\le C(1+|\xi|)^{m-r}M^m_r(a\varrho),\quad \forall \xi\in \xR^d.
\end{align*}
Estimates for $|\partial^\alpha_\xi a_j^2|$ can be derived along the same lines. Thus, $a_j^2\in\Gamma^{m-r}_0$ and hence $a_j\in \Gamma^{m-r}_0$; moreover 
\[
 M^{m-r}_0(a_j)\le CM^{m}_r(a\varrho).
\]
\hk 2.  
Property \eqref{chi:prop} implies in particular that
\[
\mathfrak{F}_x(\sigma_a)(\eta, \xi)=0~ \text{for}~|\eta|\ge \eps_2(1+|\xi|).
\]
Here, we denote $\mathfrak{F}_x$ the Fourier transform with respect the the patial variable $x$.\\
On the other hand, by definition of the smoothing operator 
\[
\mathfrak{F}_x S_{j-3}(a)(x, \xi)\varrho(\xi)\varphi_j(\xi)=\psi(2^{-(j-3)}\eta)\mathfrak{F}_xa(\eta, \xi)\varrho(\xi)\varphi(2^{-j}\xi)
\]
which is vanishing if $|\eta|\ge \mez(1+|\xi|)$. Indeed, if either $|\xi|> 2^{j+1}$ or $|\xi|\le 2^{j-1}$ then $\varphi(2^{-j}\xi)=0$. Considering $2^{j-1}<|\xi|\le 2^{j+1}$ then $|\eta|\ge \mez(1+|\xi|)>2^{j-2}$ and thus $\psi(2^{-(j-3)}\eta)=0$. We have proved the existence of $0<\eps<1$ such that 
\bq\label{spectral:a_j}
\mathfrak{F}_x a_j(\eta, \xi)=0 ~\text{for}~|\eta|\ge \eps(1+|\xi|).
\eq
\hk 3. By the spectral property \eqref{spectral:a_j} one can use the Bernstein's inequalities (see Corollary $4.1.7$, \cite{MePise}) to prove
%\[
%\forall (\alpha, \beta)\in \xN^d\times\xN^d,~\forall \xi\in \xR^d,~\la \partial_x^\alpha\partial_\xi^\beta a_j(x, \xi)\ra \le C_{\alpha, \beta}(1+|\xi|)^{m-r+|\alpha|-|\beta|}
%\]
  that $a_j$ is a pseudo-differential symbol in the class $S^{m-r}_{1,1}$. Then, applying Theorem $4.3.5$ in \cite{MePise} we conclude that 
\[
\lA R_ju\rA_{H^{\mu-m+r}(\xR^d)}=\lA \Op(a_j)u\rA_{H^{\mu-m+r}(\xR^d)}\le CM_0^{m-r}(a_j)\lA u\rA_{H^\mu(\xR^d)}.
\]
Finally, the Fourier transform of $R_ju$  reads
\[
\mathfrak{F}(R_ju)(\xi)=\int_{\xR^d}\mathfrak{F}_x(a_j)(\xi-\eta, \eta)\hat{u}(\eta)d \eta.
\]
Using the spectral localization property \eqref{spectral:a_j} and the fact that $\mathfrak{F}_x(a_j)(\xi-\eta, \eta)$ contains the factor $\varphi_j(\eta)$ we conclude that the spectrum of $R_ju$ is contained in an annulus of size $~2^j$ as claimed.
%\[
%\left\{\xi\in \xR^d: c_12^{j-1}\le |\xi|\le c_22^{j+1} \right\}.
%\]
\end{proof}
\subsection{A result  of Koch-Tataru}
In this paragraph, we recall the dispersive estimates proved by Koch-Tataru \cite{KT} based on the technique of FBI transform on phase space. These estimates were established for the following class of operators.
\begin{defi}\label{defi:symbolclass}
For $\ld >1$, $m\in \xR$ and $k=0, 1, ...$ consider classes of symbols $p:T^*\xR^d\to \xC$, denoted by $\ld ^mS^k_\ld$, which satisfy
\bq\label{symbolKT}
\begin{gathered}
\la \partial_x^\alpha\partial_\xi^\beta p( x, \xi)\ra\le c_{\alpha,\beta}\ld^{m-|\beta|}\quad |\alpha|\le k,\\
\la \partial_x^\alpha\partial_\xi^\beta p( x, \xi)\ra\le c_{\alpha,\beta}\ld^{m+\frac{|\alpha|-k}{2}-|\beta|} \quad |\alpha|\ge k.
\end{gathered}
\eq
\end{defi}
%For any symbol $p$, we set 
%\[
%B_\ld=\{|\xi|\le \ld\},~\Sigma=\text{char}p\cap B_\ld.
%\]
The mentioned result reads
\begin{prop}[\protect{\cite[Proposition 4.7]{KT}}]\label{theo:KT}
 Let $p(\sigma, x, \xi)\in \ld S^2_\ld$ be a real symbol in $(x, \xi)$, uniformly in $\sigma\in [0, 1]$. Assume that $p$ satisfies the following curvature condition\\
({\bf A}) for each $(\sigma, x, \xi)\in [0, 1]\times\xR^d\times \Cr_\ld$,  $|\det \partial_\xi^2 p|\gtrsim \ld^{-d}$, where $\Cr_\ld=\{c^{-1}\ld\le |\xi|\le c\ld\}$.\\
Denote by $S(\sigma, \sigma_0)$ the flow maps of $D_\sigma +\Op^w(p)$. Then for any $\chi\in S^0_\ld$ such that  for all $x\in \xR^d$, $\chi(x, \cdot)$ compactly supported in $\Cr'_\ld=\{c'^{-1}\ld\le |\xi|\le c'\ld\}$with $1<c'<c$, we have
\[
\lA S(\sigma, \sigma_0)(\chi(x, D_x)v_0)\rA_{L^\infty}\les \ld^{\frac{d}{2}}|\sigma-\sigma_0|^{\frac{-d}{2}}\lA v_0\rA_{L^1}\quad \forall \sigma,~\sigma_0\in [0, 1].
\]
\end{prop}
\begin{rema}
In the statement of Proposition $4.7$, \cite{KT}, condition $(\bf A)$ is stated for   $(x, \xi)\in B_\ld:=\{|x|\le 1,~|\xi|\le \ld\}$ and correspondingly, $\chi$ is supported in $B$; in addition, the usual quantization $\chi(x, D_x)$ is replaced by the Weyl quantization $\chi^w$. However, one can  easily inspect its proof to see that if $(\bf A)$ is fulfilled globally in $x$ then we have the above variant.
\end{rema}
%Concerning complex-valued symbols, the next result is a corollary of Theorem $3$, \cite{KT}.
%\begin{theo}\label{theo:KT1}
%Let $p(\sigma, x, \xi), ~q(\sigma, x, \xi)$ be two real symbols such that
%\[
%\partial_\sigma+\{p, q\}\in S^0_\ld
%\]
%in $(x, \xi)$, uniformly in $\sigma\in [0, 1]$. Assume that $p$ satisfies the curvature condition $(\bf A)$. \\
%Suppose that $u$ is a solution of
% \[
%D_\sigma u+\Op^w(p+iq)u=0,\quad u_{t=0}=u^0.
%\]
% Then for any $\chi\in S^0_\ld$ such that  for all $x\in \xR^d$, $\chi(x, \cdot)$ compactly supported in $\Cr'_\ld=\{c'^{-1}\ld\le |\xi|\le c'\ld\}$with $1<c'<c$, we have
%\[
%\lA \chi^w u\rA_{L^qL^r}\les \ld^{\frac{1}{q}}\lA u^0\rA_{L^2},
%\]
%where
%\[
%\frac{2}{q}+\frac{d}{r}=\frac{d}{2},~2\le q,~r\le \infty,~(q,r)\ne (2, \infty).
%\]
%\end{theo}
\subsection{Remarks on the symbolic calculus for $\lambda^mS^k_\lambda$}
Let $a\in \lambda^mS^k_m,~k\ge 0$ and suppose that on the support of $a(x,\xi)$,~$\lambda^{-1}|\xi|\sim 1$. It follows by definition of $\lambda^mS^k_\lambda$ that $a\in S^m_{1,\mez}$. Observe however that when $k\ge 1$, $a$ behaves better than a general symbol in  $S^m_{1,\mez}$. In this paragraph we present some enhanced properties of $\lambda^mS^k_\lambda$ with $k\ge 1$.\\
First, we are concerned with the relation between the the Weyl quantization and the usual quantization. According to Proposition \ref{expand}, for $a\in S^m_{1, \mez}$ there holds
\[
\Op^w(a)-\Op(a)=\Op(r),\quad r\in S^{m-\mez}_{1, \mez}.
\]
In fact, we have
\[
\Op^w(a)=\Op(\widetilde a),\quad \widetilde a(x, y, \xi)=a(\frac{x+y}{2}, \xi)
\]
and
\[
 \widetilde a(x, y, \xi)=a(x+\frac{y-x}{2}, \xi)=a(x, \xi)+\mez\int_0^1(\partial_xa)(x+s\frac{y-x}{2}, \xi)ds\, (y-x).
\]
It follows that
\bq\label{toW}
\Op^w(a)-\Op(a)=\Op(r),
\eq
with
\bq\label{toW:remainder}
r(x, y, \xi)=-\frac{i}{2}\int_0^1(\partial_\xi\partial_x a)(x+s\frac{y-x}{2}, \xi)ds.
\eq
We now show that in fact $r$ is of order $m-1$ as in the case $a\in S^m_{1,0}$.
\begin{lemm}\label{toW:improve}
Let $a\in \lambda^mS^k_\ld,~k\ge 1$ satisfying $\lambda^{-1}|\xi|\sim 1$ on the support of $a(x,\xi)$. Then we have the relation \eqref{toW}-\eqref{toW:remainder} with $r\in S^{m-1}_{1,\mez, \mez}$.
\end{lemm}
\begin{proof}
For all $\alpha, \beta, \nu\in \xN^d$ we have
\[
|\partial_x^\alpha\partial_y^\beta \partial_\xi^\nu r(x, y, \xi)|\le 
\begin{cases}
&C_{\alpha, \beta, \nu}\lambda^{m-|\nu|-1}\mathbbm{1}_{\lambda^{-1}|\xi|\sim 1}(\xi)\quad\text{if}~|\alpha|+|\beta|+1\le k,\\
&C_{\alpha, \beta, \nu}\lambda^{m+\frac{|\alpha|+1+|\beta|-k}{2}-|\nu|-1}\mathbbm{1}_{\ld^{-1}|\xi|\sim 1}(\xi)\quad\text{if}~|\alpha|+|\beta|+1\ge k.
\end{cases}
\]
Since $k\ge 1$,
\[
\frac{|\alpha|+1+|\beta|-k}{2}\le \frac{|\alpha|+|\beta|}{2}
\]
Consequently, it holds that
\[
|\partial_x^\alpha\partial_y^\beta \partial_\xi^\nu r(x, y, \xi)|\le C_{\alpha, \beta, \nu}(1+|\xi|)^{(m-1)+\frac{|\alpha|+|\beta|}{2}-|\nu|}\quad \forall \alpha, \beta, \nu\in \xN^d.
\]
In other words,  $r\in S^{m-1}_{1,\mez, \mez}$. 
\end{proof}
For the composition rule we prove the following lemma.
\begin{lemm}\label{compose:improve}
Let $p\in S^n_{1, 0}$ and $a\in \lambda^mS^k_\ld,~k\ge 1$ satisfying $\lambda^{-1}|\xi|\sim 1$ on the support of $a(x,\xi)$. Then we have
\[
\Op(p)\circ \Op(a)-\Op(pa)=\Op(r)
\]
with $r\in S^{n+m-1}_{1, \mez}$.
\end{lemm}
\begin{proof}
According to Proposition 0.3.C, \cite{Taylor}, one has $\Op(p)\circ \Op(a)=\Op(b)$ with 
\[
b\sim \sum_{|\alpha|\ge 0}\frac{(-i)^{|\alpha|}}{\alpha!}\partial_\xi^\alpha p(x, \xi)\cdot\partial_x^\alpha a (x, \xi)
\]
in the sense of symbol asymptotic. The general term in the above expansion belongs to $S^{n+m-\frac{|\alpha|}{2}}_{1, \mez}$, hence 
\[
(b-pa)-\sum_{|\alpha|= 1}\frac{(-i)^{|\alpha|}}{\alpha!}\partial_\xi^\alpha p(x, \xi) \cdot\partial_x^\alpha a (x, \xi)\in S^{n+m-1}_{1, \mez}.
\]
It then suffices to show that $c_\alpha:=\partial_\xi^\alpha p(x, \xi)\cdot \partial_x^\alpha a (x, \xi)\in S^{n+m-1}_{1, \mez}$ for $|\alpha|=1$ or again, $\partial_x^\alpha a (x, \xi)\in S^m_{1, \mez}$ for $|\alpha|=1$. The later follows along the same lines as in the proof of Lemma \ref{toW:improve}.
\end{proof} 
In the same spirit we have the following result on adjoint operators, taking into account  Proposition 0.3.B, \cite{Taylor}.
\begin{lemm}\label{adjoint:improve}
Let $a\in \lambda^mS^k_\ld,~k\ge 1$ satisfying $\lambda^{-1}|\xi|\sim 1$ on the support of $a(x,\xi)$. Then we have
\[
\Op^*(a)-\Op(\bar a)=\Op(r)
\]
with $r\in S^{m-1}_{1, \mez}$ and $\bar a$ is the complex conjugate of $a$.
\end{lemm}
\begin{nota}
 Throughout this article, we write $A\les B$ if there exists a constant $C>0$ such that $A\le CB$, where $C$ may depend on the coefficients of the equations under consideration. If the constant $C$ involved has some explicit dependency, say, on some quantity $\mu$, we emphasize this by denoting $A\les_\mu B$.
\end{nota}
%%%%%%%%%%%%%%%%%%%%%%%%%%
\section{Proof of Theorem \ref{main:theo:0}}
Throughout this section, the dimension $d$ is greater than or equal $2$ and $s>\frac{5}{2}+\frac{d}{2}$ is a Sobolev index.
\subsection{Littlewood-Paley reduction}
We shall prove Strichartz estimates for solution $u$ to \eqref{ww:reduce:0}, which is a quasilinear paradifferential equation with time-dependent coefficients. Remark that since
\[
(\eta, \psi)\in C^0([0, T]; H^{s+\mez}\times H^s),
\]
we have
\[
V\in C^0([0, T]; H^{s-1}),\quad \gamma(\cdot, \xi)\in C^0([0, T]; H^{s-\mez}).
\]
The first step in our proof consists in localizing \eqref{ww:reduce:0} at frequency $2^j$ using the Littlewood-Paley  decomposition (cf. Definition \ref{functionalspaces} 1.).  For every $j\geq 0$, the dyadic piece~$\Delta_ju$ solves 
\begin{equation}\label{eq:dyadic:0}
	\lp\partial_t+T_V\cdot\nabla +iT_{\gamma+\omega} \rp\Delta_ju=F_j,
\end{equation}
with
\begin{equation}\label{Fj1:0}
	F^1_j:=\Delta_jf -i\Delta_j (T_{\omega_1} u)+i\lb T_\gamma,\Delta_j\rb+i\lb T_\omega,\Delta_j\rb u+\lb T_V,\Delta_j\rb\cdot\nabla u.
\end{equation}
Remark that for each $j\ge 1$, $\Delta_j u$ is spectrally supported in $\{2^{j-1}\le |\xi|\le 2^{j+1}\}$. In views of  Proposition \ref{para:pseudo} and the facts that  $\gamma\in \Gamma^{\tdm}_{\tdm},~\omega\in \Gamma^{\mez}_{\mez}$ and $V\cdot \xi\in \Gamma^1_1$, equation \eqref{eq:dyadic:0} can be rewritten as
\begin{equation} \label{eq:pseudo:0}
	\lp\partial_t+S_{j-3}(V)\cdot\nabla+iS_{j-3}(\gamma)(x, D_x) \rp\Delta_ju=F^2_j,
\end{equation}
with
\bq\label{Fj2:0}
F^2_j:=F^1_j +R_j,
\eq
$R_ju$ is spectrally supported in an annulus $\{ c_1^{-1}2^{j-1}\le |\xi|\le c_12^{j+1}\}$ and satisfies 
\bq\label{bound:R:0j}
\lA R_j\rA_{H^s}\les \lA u\rA_{H^s}.
\eq
Next, we use \eqref{smoothout} to smooth out the symbols by $\delta=\mez$ derivative.\\
Now, let $\mez<c_1<c_2<c_3$, $\Cr^k:=\{(2c_k)^{-1}\le |\xi|\le 2c_k\},~k=1,2,3$ and 
\[
\tph\in C^\infty,~\supp\tph\subset \Cr^3,~\tph\equiv 1~\text{on}~\Cr^2.
\]
Then, equation~(\ref{eq:pseudo:0}) is equivalent to
\begin{multline} \label{eq:reg:0}
	\mathcal{L}_j\Delta_ju:=\lp\partial_t+ S_{{(j-3)}\mez}(V)\cdot\nabla\tph(2^{-j}D_x)\right. \\
\left.+iS_{{(j-3)}\mez}(\gamma)(x, D_x)\tph(2^{-j}D_x)+iS_{{(j-3)}\mez}(\omega)(x, D_x)\tph(2^{-j}D_x)\rp\Delta_ju=F_j,
\end{multline}
 with
\begin{multline}	\label{Fj:0}
	F_j=F_j^2+F_j^3:=F_j^2+i\lp S_{{(j-3)}\mez}\gamma(x,D_x)-S_{j-3}\gamma(x, D_x)\rp\Delta_ju\\+i\lp S_{{(j-3)}\mez}\omega(x,D_x)-S_{j-3}\omega(x, D_x)\rp\Delta_ju 
	      +\lp S_{{(j-3)}\mez}(V)-S_{j-3}(V)\rp\cdot\nabla \Delta_ju.
\end{multline}
\subsection{Semi-classical time scale}
Observe that the highest order operator on the left-hand side of \eqref{eq:pseudo:0} has order $\tdm$, which does not match the result in \cite{KT} that we want to apply. Therefore, we reduce it to  an operator of order $1$ by multiplying both side by $h^{\mez}$, where 
\[
h:=2^{-j},
\] 
then making a change of temporal variables $t=h^\mez \sigma$. For this purpose, let us reset the symbols in this new time scale:
\begin{gather}
\Gamma_h(\sigma, x, \xi)=h^\mez S_{{(j-3)}\mez}(\gamma)(h^\mez \sigma, x, \xi)\tph(h\xi),\\
\Omega_h(\sigma, x, \xi)=h^\mez S_{{(j-3)}\mez}(\omega)(h^\mez \sigma, x, \xi)\tph(h\xi),\\
V_h(\sigma, x, \xi)=h^\mez S_{{(j-3)}\mez}(V)(h^\mez \sigma, x)\cdot\xi\tph(h\xi).
\end{gather}
Next, set 
\[
w_h(\sigma, x)=\Delta_j u(h^\mez \sigma, x),\quad G_h(\sigma, x)=-ih^\mez F_j(h^\mez\sigma, x). 
\]
Equation \eqref{eq:reg:0} is then equivalent to 
\bq\label{eq:wh:0}
\lp D_t+\Gamma_h(\sigma, x, D_x)+\Omega_h(\sigma, x, D_x)+V_h(\sigma, x, D_x)\rp w_h(\sigma, x)= G_h(\sigma, x).
\eq
In what follows, we shall prove the classical Strichartz estimate for \eqref{eq:wh:0}, from which the  semi-classical Strichartz estimate for \eqref{eq:reg:0} follows.\\
We now replace the pseudo-differential operators in \eqref{eq:wh:0} by the corresponding  Weyl operators using Proposition \ref{expand}. Noticing that $\omega=-\frac{i}{2}\partial_\xi\cdot\partial_x\gamma$, we set
\bq\label{R1R2}
\begin{aligned}
&R_h^1:=\big(\Op(\Gamma_h)+\Op(\Omega_h)\big)-\Op^w(\Gamma_h),\\
&R_h^2:=\Op(V_h)-\Op^w(V_h).
\end{aligned}
\eq
With these notations \eqref{eq:wh:0} becomes
\bq\label{eq:whw:0}
\begin{aligned}
L_hw_h(\sigma, x)&:=\lp D_t+\Gamma_h^w(\sigma, x, D_x)+V_h^w(\sigma, x, D_x)\rp w_h(\sigma, x)\\
&= G_h(\sigma, x)+R^1_h(\sigma, x)w_h(\sigma, x)+R^2_h(\sigma, x)w_h(\sigma, x).
\end{aligned}
\eq
%\begin{prop}
%We have 
%\[
%\Op^w(\Gamma_h)=\Op(\Gamma_h)+\Op(\Omega_h)+\sum_{|\alpha|\ge 2}\frac{(-i)^{|\alpha|}}{\alpha!2^{|\alpha|}}\Op(\partial_x^\alpha\cdot\partial_\xi^\alpha \Gamma_h)
%\]
%in the sense that when acting on functions, the right-hand  both sides are equal in any Sobolev spaces 
%\end{prop}
\subsection{Classical Strichartz estimate for \eqref{eq:wh:0} ($\Leftrightarrow$ \eqref{eq:whw:0})}
In this step, we shall show that Theorem \ref{theo:KT} can be applied to the real symbol
\[
p_h:=\Gamma_h+V_h.
\]
Set $\ld=h^{-1}=2^j.$ First, the characteristic set of $\gamma$ has maximal ($d$) nonvanishing principal curvatures:
\begin{prop}
Let $\Cr$ be a fixed annulus in $\xR^d$.\\
1. There exists an absolute constant $C_d>0$ such that  with~$c_0=C_d(1+ \lA\nabla\eta \rA_{L^\infty(I\times \xR^d)})$  we have 
\bq\label{Hessgamma}
		      	\sup_{(t, x, \xi)\in I\times \xR^d\times \Cr}\la\det\partial_\xi^2\gamma(t,x,\xi)\ra\geq c_0.	
\eq
2. For any $0<\delta\le 1$ there exists $j_0\in \xN$ large enough such that 
\bq\label{HessGamma}
		      	\sup_{(t, x, \xi)\in I\times \xR^d\times \Cr}\la\det\partial_\xi^2S_{j\delta}(\gamma)(t,x,\xi)\ra\geq c_0.	
\eq
\end{prop}
\begin{proof}
1. For the proof of part 1, we refer to  Corollary $4.7$, \cite{ABZ4}. Part 2. is a consequence of part 1. because $S_{j\delta}(\gamma)$ is a small perturbation of $\gamma$ when $j$ is large enough (see for instance Proposition $4.5$, \cite{ABZ4}).
\end{proof}
\begin{lemm}\label{check:a}
1. We have $\Gamma_h\in \ld S^2_\ld,~V_h\in \ld^{\frac{3}{4}}S^2_\ld$ and hence $p_h\in \ld S^2_\ld$.\\
2.  There exists $h_0>0$ small enough such that for $0<h\le h_0$, the symbol $p_h$ satisfies condition ({\bf A}) in Theorem \ref{theo:KT} with  $\Cr_\ld=\ld\Cr^2$.
\end{lemm}
\begin{proof}
1. Since $\gamma\in W^{2,\infty}$ and $V\in W^{\tdm,\infty}$ in $x$, assertion 1. then follows easily from Lemma \ref{Bernstein} and the fact that on the support of $\tph(h\xi)$ we have $|\xi|\sim \ld$.\\
2. %Set $B_\ld=\{|\xi|\le c_3\ld\}$ with $c_3>c_2$.  Remark that on $B_\ld$ every derivative of $\tph$ vanishes.  With $V_h(\sigma, x, \xi)=h^\mez S_{{(j-3)}\mez}(V)(h^\mez \sigma, x)\cdot\xi\tph(h\xi)$ we see that for every $\xi\in B_\ld$, in the expression of $\partial_\xi^2V_h$, the function $\tph$ is differentiated at least once and thus it second fundamental form vanishes in $B_\ld$. For $\Gamma_h$ one gets
Let $\xi \in\ld\Cr^2$. We have $\tph(\th\xi)= 1$,  hence $\partial_\xi^2V_h$ vanishes and since $\gamma$ is homogeneous of order $\tdm$ in $\xi$,
\[
\partial_\xi^2\Gamma_h(\sigma, x, \xi)= \partial_\xi^2\lp  h^\mez S_{{(j-3)}\mez}(\gamma)(h^\mez \sigma, x, \xi)\rp=h \lp\partial_\xi^2S_{{(j-3)}\mez}(\gamma)\rp(h^\mez \sigma, x, h\xi).
\]
Therefore, condition $({\bf A})$ is verified by virtue of \eqref{HessGamma}.
\end{proof}
Calling $S_h(\sigma, \sigma_0)$ the flow map of the evolution operator $L_h=D_\sigma+\Op^w(p_h)$ (see \eqref{eq:whw:0}), we have:
\begin{lemm}\label{dis:w}
%	Let~$\kappa\in C^\infty_c(\xR^d)$ be supported in~$\mathcal{C}_1:=\lB\xi:c_1^{-1}2^{-1}\leq\la\xi\ra\leq c_12\rB$. Take~$\sigma_0\in\xR$, $v^0\in L^1(\xR^d)$ and set~$v_h^0:=\kappa(hD_y)v_0$.
If $v_h^0$ is spectrally supported in  $\ld\mathcal{C}^1=\lB (2c_1)^{-1}h^{-1}\leq\la\xi\ra\leq 2c_1h^{-1}\rB$ then\\
$(i)$ \[
		\lA S_h(\sigma,\sigma_0)v^0_h\rA_{L^\infty(\xR^d)}\les  h^{-\frac{d}{2}}
		  \la\sigma-\sigma_0\ra^{-\frac{d}{2}}\lA v_h^0\rA_{L^1(\xR^d)}
\]
	for all~$\sigma,~\sigma_0\in [0, 1]$ and~$0<h\leq h_0$;\\
$(ii)$ with $q>2$ and~$\frac{2}{q}=\frac{d}{2}-\frac{d}{r}$, 
	\begin{equation*}
		\lA S_h(\cdot, 0)v^0_h\rA_{L^q([0, 1],L^r)}\les  h^\frac{-1}{q}\lA v_h^0\rA_{L^2}.
	\end{equation*}
\end{lemm}
\begin{proof}
By Lemma \ref{check:a}, $(i)$ is a direct consequence of Proposition \ref{theo:KT} if one chooses 
\[
\chi(\xi)\in C^\infty,~ \supp\chi\subset \{(2c_{1,2})^{-1}\ld\le |\xi|\le 2c_{1,2}\ld\},~c_1<c_{1,2}<c_2,~\chi\equiv 1 ~\text{in}~\Cr^1.
\]
 For $(ii)$ we remark that since $\Op^w(\Gamma_h)$ and $\Op^w(V_h)$ are self-adjoint, $S_h(\sigma, \sigma_0)$ is isometric in $L^2$. This combining with the dispersive estimate $(i)$ and a standard $TT^*$~argument (see the abstract semi-classical Strichartz estimate in Theorem $10.7$, \cite{Zworski}) yields $(ii)$.
\end{proof}
\begin{lemm}\label{energy:0}
For any $\mu \in \xR$, the operators $S_h(\sigma, \tau)$ are bounded on $H^\mu(\xR^d)$ uniformly in $\tau, \sigma\in I$ .
\end{lemm}
\begin{proof}
This result bases on a standard energy estimate. However, the proof requires more care  since we are not working on standard operators of classes $S^m_{1,0}$. Without loss of generality we assume $\tau=0$ and let $f(\sigma, x)$ be a solution of 
\[
\lp \partial_\sigma+i\Op^w(p_h)\rp f(\sigma, x)=0,\quad f(0, x)=f_0(x).
\]
We first apply Lemma  \ref{toW:improve} to obtain
\[
\Op^w(p_h)=\Op(p_h)+\Op(r_h),\quad r_h\in S^{0}_{1, \mez,\mez}.
\]
Then $f$ solves the problem
\[
\lp \partial_\sigma+i\Op(p_h)+i\Op(r_h)\rp f(\sigma, x) =0, \quad f(0, x)=f_0(x).
\]
Let $\mu \in \xR$ and set $f^\mu:=\langle D_x\rangle ^\mu f$. Then
\[
\frac{d}{d\sigma} \lA f^\mu\rA_{L^2}^2 =-i\langle \lp \Op(p_h)-\Op^*(p_h)\rp f^\mu, f^\mu\rangle +2\Re\langle F, f^\mu\rangle 
\]
where 
\[
F:=-i\lb  \langle D_x, \Op(p_h)\rangle^\mu \rb f-i\langle D_x\rangle^\mu \Op(r_h) f.
\]
According to Lemma \ref{compose:improve}, $\lb \langle D_x\rangle^\mu, \Op(p_h) \rb \in S^\mu_{1, \mez}$. This combining with the fact that $r_h\in S^{0}_{1, \mez,\mez}$ gives
\[
\lA F\rA_{L^2}\les \lA f\rA_{H^\mu}.
\]
 On the other hand,  since  $p_h$ is real, Lemma \ref{adjoint:improve} implies 
\[
\Op(p_h)-\Op^*(p_h)\in S^0_{1, \mez}.
\]
Consequently, we have
\[
\lA \lp\Op(p^0_h)-\Op^*(p^0_h)\rp f^\mu\rA_{L^2}\les \lA f^\mu\rA_{L^2}.
\]
Finally, we conclude by Gronwall's inequality that
\[
\lA f(\sigma)\rA_{H^\mu}\les \lA f_0\rA_{H^\mu}\quad \forall \sigma\in I.
\]
\end{proof}
\begin{prop}\label{Str:w}
If $w_h$ is a solution to \eqref{eq:wh:0} with $w_h(0)=w^0_h$ and
\[
 \supp\widehat {w_h},~\supp\widehat {w^0_h}\subset \ld\Cr^1
\]
 then we have for all $\eps>0$
\[
		\lA w_h\rA_{L^{2+\eps}([0, 1], W^{s-\frac{d}{2}+\mez-\eps,\infty})}\les_\eps\lA w^0_h\rA_{H^s}+\lA G_h\rA_{L^1([0, 1], H^s)}+h^\mez\lA w_h\rA_{L^1([0, 1], H^s)}.
\]
\end{prop}
\begin{proof}
To simplify notations, we write $S_h(\sigma, \tau)=S(\sigma, \tau)$. If $w_h$ is a solution to \eqref{eq:wh:0}, it is also a solution to \eqref{eq:whw:0}. By Duhamel's formula, there holds
\[
w_h(\sigma, 0)=S_h(\sigma, 0)w_h^0+\int_0^\sigma S(\sigma, \tau)[G_h(\tau)]d\tau+\int_0^\sigma S(\sigma, \tau)[(R^1_hw_h+R^2_hw_h)(\tau)]d\tau.
\]
Let us call $(I)$ and $(II)$, respectively, the first and the second integral on the right-hand side. Choosing $c_1$ large enough such that $\supp\widehat G_h\subset \ld\Cr^1$, Lemma \ref{dis:w} $(ii)$ gives 
\[
\lA (I)\rA_{L^q([0, 1], L^r)}\les h^{-\frac{1}{q}}\lA G_h\rA_{L^1([0, 1]; L^2)}
\]
For $(II)$ we set 
\[
b^\alpha_h=\frac{(-i)^{|\alpha|}}{\alpha!2^{|\alpha|}}\partial_x^\alpha\partial_\xi^\alpha \Gamma_h (\sigma, x, \xi)~ |\alpha|\ge 2; \quad c^\alpha_h= \quad \frac{(-i)^{|\alpha|}}{\alpha!2^{|\alpha|}}\partial_x^\alpha\partial_\xi^\alpha V_h(\sigma, x, \xi)~|\alpha|\ge 1.
\]
For each $|\alpha|\ge 2$, since  $\gamma$ is $W^{2,\infty}$ in $x$ we have by applying Lemma \ref{Bernstein}
\[
\lA \partial^\mu_x\partial_\xi^\nu (\partial_x^\alpha\partial_\xi^\alpha \Gamma_h)\rA \les  (1+|\xi|)^{1+\frac{|\mu|+|\alpha|-2}{2}-(|\nu|+|\alpha|)}\les (1+|\xi|)^{-\frac{|\alpha|}{2}+\frac{|\mu|}{2}-|\nu|},
\] 
hence $b^\alpha_h\in S^{-\frac{|\alpha|}{2}}_{1,\mez}$. Similarly, as $V\in W^{1, \infty}$ it holds that $c^\alpha_h\in S^{-\frac{|\alpha|}{2}}_{1,\mez}$ for $|\alpha|\ge 1$ . %Moreover, on the support of $\tph$, $|\xi|\sim h^{-1}$ so we also have 
%\bq\label{est:bc}
%h^{1-\frac{|\alpha|}{2}}b^\alpha_h\in S^0_{1,\mez},\quad h^{\mez-\frac{|\alpha|}{2}}c^\alpha_h\in S^0_{1,\mez}.
%\eq
Taking $q>2$ and $\frac{2}{q}=\frac{d}{2}-\frac{d}{r}$, we claim that uniformly in $\tau\in [0, 1]$
\bq\label{claim:R}
\lA S(\sigma, \tau)[(R^1_hw_h)(\tau)]\rA_{L^q_\sigma L^r_x}\les h^{-\frac{1}{q}+1}\lA w_h(\tau)\rA_{L^2_x}.
\eq
Indeed, by the asymptotic expansion in  Proposition \ref{expand},
\[
R_h^1=\sum_{|\alpha|=2}^{N-1}\Op(b_h^\alpha)+\Op(r^N_h),~r_h^N\in S^{1-\frac{N}{2}}_{1,\mez}\quad \forall N\ge 3.
\]
If we choose $c_1$ large enough then each $\Op(b^\alpha_hw_h)(\tau)$  is spectrally supported in $\ld\mathcal{C}^1$ (and so is $w_h$) so that Lemma \ref{dis:w} $(ii)$ can be applied  to get
\bq\label{est:R}
\| S(\sigma, \tau)[\Op(b^\alpha_h)w_h(\tau)]\|_{L^q_\sigma L^r_x}\le h^{-\frac{1}{q}}\lA \Op(b^\alpha_h)w_h\rA_{L^2_x}\les h^{-\frac{1}{q}+\frac{|\alpha|}{2}}\lA w_h(\tau)\rA_{L^2_x}.
\eq
For $\Op(r^N_h)$ one uses the Sobolev embedding $H^{\frac{d}{2}}\hookrightarrow L^r,~\forall r\in [2, +\infty)$ to estimate 
\[
\| S(\sigma, \tau)[\Op(r^N_h)w_h(\tau)]\|_{L^r_x}\les \| S(\sigma, \tau)[\Op(r^N_h)w_h(\tau)]\|_{H^{\frac{d}{2}}_x}.
\]
On the other hand, we know from Lemma \ref{energy:0}  that $S(\sigma, \tau)$ is bounded from $H^\mu$ to $H^\mu$ uniformly in $\sigma,~\tau\in [0, 1]$ for all $\mu \in \xR$. Hence
\bq\label{est:R1}
\| S(\sigma, \tau)[\Op(r^N_h)w_h(\tau)]\|_{L^r_x}\les h^{-1+\frac{N}{2}-\frac{d}{2}}\lA w_h(\tau)\rA_{L^2_x}.
\eq
Choosing $N=N(d)$ large enough, we conclude the claim \eqref{claim:R} from \eqref{est:R} and \eqref{est:R1}.\\
 In the same way, we obtain the following estimate for $R^2_h$ (which is also uniformly in $\tau\in [0, 1]$)
\[
\lA S(\sigma, \tau)[(R^2_hw_h)(\tau)]\rA_{L^q_\sigma L^r_x}\les h^{-\frac{1}{q}+\mez}\lA w_h(\tau)\rA_{L^2_x}.
\]
Putting together the above estimates  leads to
\[
\lA w_h\rA_{L^qL^r}\le h^{-\frac{1}{q}}\lp \lA w^0_h\rA_{L^2}+\lA G_h\rA_{L^1([0, 1]; L^2)}+h^\mez\lA w_h\rA_{L^1L^2}\rp.
\]
%On the other hand, the energy estimate for \eqref{eq:wh} implies that the last right-hand side term is absorbed by the first two terms, hence
%\[
%\lA w_h\rA_{L^qL^r}\le h^{-\frac{1}{q}}\lp \lA w^0_h\rA_{L^2}+h^\mez\lA G_h\rA_{L^1([0, 1]; L^2)}\rp.
%\]
%this together with \eqref{est:bc} implies 
%\[
%\| S(\cdot, \tau)\Op(b^\alpha_hw_h)(\tau)\|_{L^qL^r}\les  h^{-\frac{1}{q}-1+\frac{|\alpha|}{2}}\lA w_h(\tau)\rA_{L^2}
%\]
Taking~$q=2+\eps$ then ~$h^\frac{-1}{q}\leq h^\frac{-1}{2}$. We multiply both sides by~$h^{-s}$ and use the frequency localization of~$w_h,~w_h^0$ to get
	\begin{equation*}
		\lA w_h\rA_{L^{2+\eps}W^{s-\frac{1}{2},r}}\les_\eps \lA w^0_h\rA_{H^s}+\lA G_h\rA_{L^1 H^s}+h^\mez\lA w_h\rA_{L^1H^s}.
	\end{equation*}
	We write ~$s-\frac{1}{2}=(\frac{d}{2}-1+\eps)+(s-\frac{d}{2}+\mez-\eps)$ where $\frac{d}{2}-1+\eps>\frac{d}{2}-1+\frac{\eps}{2+\eps}=\frac{d}{r}$.  The Sobolev embedding  $W^{s-\mez, r}\hookrightarrow W^{s-\frac{d}{2}+\mez-\eps,\infty}$ then concludes the proof.
\end{proof}
\subsection{Semi-classical Strichartz estimate for \eqref{eq:reg:0}}
From the preceding Proposition, one deduces the corresponding Strichartz estimate for $u_j\equiv \Delta_ju$ as a solution of \eqref{eq:reg:0} via the change of temporal variables $w_h(\sigma, x)=u_j(h^\mez\sigma, x)$ as follows.
\begin{coro}\label{Str:short}
If $u_j$ is solution to \eqref{eq:reg:0}, {\it i.e.}, $\mathcal{L}_ju_j=F_j$ with data $u_j^0$ and $u_j,~u_j^0, F_j$ are spectrally supported in $2^j\Cr^1$ then $u_j$ satisfies with $I_j=[0, 2^{-\frac{j}{2}}]= [0, h^\mez]$ and $\eps>0$ 
\[
\lA u_j\rA_{L^{2+\eps}(I_j;W^{s-\frac{d}{2}+\frac{3}{4}-\eps,\infty})}\les_\eps \lA u_j^0\rA_{H^s}+\lA F_j\rA_{L^1(I_j; H^s)}+\lA u_j\rA_{L^1H^s}.
\]
\end{coro}
The next step consists in gluing the estimates on small time scales above to derive an estimate on the whole interval of time $[0, 1]$ to the price of loosing $\frac{1}{4}$ derivatives.
\begin{coro}\label{Str:long}
If $u_j$ is solution to $\mathcal{L}_ju_j=F_j$  with data $u_j^0$ and $u_j,~u_j^0,~F_j$ are spectrally supported in $\Cr^1_h$ then $u_j$ satisfies with $I=[0, T]$ and $\eps>0$
\[
\lA u_j\rA_{L^{2}(I;W^{s-\frac{d}{2}+\mez-\eps,\infty})}\les_\eps \lA F_j\rA_{L^2(I; H^{s-\mez})}+\lA u_j\rA_{L^\infty(I,H^s)}.
\]
\end{coro}
\begin{proof}
Let $\chi \in C_0^\infty(0,2)$ equal to one on $[ \mez, \frac{3}{2}].$ For $0 \leq k \leq [Th^{-\mez}]-2$ define
$$
I_{j,k} = [kh^\mez, (k+2)h^\mez), \quad \chi_{j,k}(t) = \chi\Big( \frac{t-kh^\mez}{h^\mez}\Big), \quad  u_{j,k} = \chi_{j,k}(t)u_j.
$$
Then each $u_{j, k}$ is a solution to
\[
\mathcal{L}_ju_{j, k}=\chi_{j,k}F_j+h^{-\mez}\chi'\Big( \frac{t-kh^\mez}{h^\mez}\Big)u_j,\quad u_{j,k}(t=kh^\mez)=0,
\]
 from which we deduce by virtue of Corollary \ref{Str:short} 
\[
\lA u_{j, k}\rA_{L^{2}(I_{j,k};W^{s-\frac{d}{2}+\frac{3}{4}-\eps,\infty})}\les_\eps \lA F_j\rA_{L^1(I_{j,k}; H^s)}+h^{-\mez}\lA u_j\rA_{L^1(I_{j,k}; H^s)}. 
\]
Notice that $\chi_{j,k}=1$ on $\lp(k+\mez)h^\mez, (k+\tdm)h^\mez\rp$ we get
\begin{align*}
\lA u_{j, k}\rA_{L^{2}((k+\mez)h^\mez, (k+\tdm)h^\mez);W^{s-\frac{d}{2}+\frac{3}{4}-\eps,\infty})}&\les_\eps \lA F_j\rA_{L^1(I_{j,k}; H^s)}+h^{-\mez}\lA u_j\rA_{L^1(I_{j,k}; H^s)}\\
&\les_\eps h^{\frac{1}{4}}\lA F_j\rA_{L^2(I_{j,k}; H^s)}+h^{-\frac{1}{4}}\lA u_j\rA_{L^2(I_{j,k}; H^s)}.
\end{align*}
Squaring both sides of the above inequality and then summing in $0 \leq k \leq [Th^{-\mez}]-2=:N_h$  yields with $J_j:=[\mez h^\mez, (N_h-\mez)h^\mez)$
\[
\lA u_j\rA_{L^{2}(J_j;W^{s-\frac{d}{2}+\frac{3}{4}-\eps,\infty})}\les_\eps h^{\frac{1}{4}}\lA F_j\rA_{L^2(I; H^s)}+h^{-\frac{1}{4}}\lA u_j\rA_{L^2(I; H^s)}
\]
or, equivalently after multiplying by $h^{\frac{1}{4}}$
\bq\label{Strl:1}
\lA u_j\rA_{L^{2}(J_j;W^{s-\frac{d}{2}+\mez-\eps,\infty})}\les_\eps \lA F_j\rA_{L^2(I; H^{s-\mez})}+\lA u_j\rA_{L^2(I; H^s)}.
\eq
On the other hand, on $J=[0, \mez h^\mez]$ one can apply Corollary \ref{Str:short} and H\"older's inequality to have
\bq\label{Strl:2}
\begin{aligned}
\lA u_j\rA_{L^{2}(J;W^{s-\frac{d}{2}+\mez-\eps,\infty})}&\les 
h^{\frac{1}{4}}\lA u_j\rA_{L^{2}(J;W^{s-\frac{d}{2}+\frac{3}{4}-\frac{\eps}{2},\infty})}\\
&\les_\eps h^{\frac{1}{4}}\lA u_j^0\rA_{H^s}+h^{\frac{1}{2}}\lA F_j\rA_{L^2(J; H^s)}+h^\mez\lA u_j\rA_{L^2(J,H^s)}\\
&\les_\eps \lA u_j\rA_{L^\infty(I, H^{s-\frac{1}{4}})}+\lA F_j\rA_{L^2(I; H^{s-\mez})}.
\end{aligned}
\eq
Likewise, we have \eqref{Strl:2} with $J=[(N_h-\mez)h^\mez, T]$ and  thus  \eqref{Strl:1} concludes that
\[
\lA u_j\rA_{L^{2}(I;W^{s-\frac{d}{2}+\mez-\eps,\infty})}\les_\eps \lA F_j\rA_{L^2(I; H^{s-\mez})}+\lA u_j\rA_{L^\infty(I,H^s)}.
\]
\end{proof}
In the final step, we shall glue the estimates for $\Delta_j u$ over different frequency regimes to obtain an estimate for $u$, from which the corresponding estimates for $(\eta, \psi)$ follow.
\subsection{Concluding the proof of Theorem \ref{main:theo:0}}
If $u$ is a solution to \eqref{ww:reduce:0} with $u(0)=u^0$ then by \eqref{eq:reg:0}, the dyadic piece $\Delta_j u$ is a solution to $L_j\Delta_j u=F_j$ with $F_j$ given by \eqref{Fj:0}. Applying Corollary \ref{Str:long} one gets
\bq\label{Str:dyadic:0}
\lA \Delta_ju\rA_{L^{2}(I;W^{s-\frac{d}{2}+\mez-\eps,\infty})}\les_\eps\lA F_j\rA_{L^2(I; H^{s-\mez})}+\lA \Delta_ju\rA_{L^\infty(I; H^s)}. 
\eq
Recall that $F_j=F_j^1+R_j+F_j^3$ where $\lA R_j\rA_{H^s}\le \lA u\rA_{H^s}$ (see \eqref{bound:R:0j}) and  $F_j^k$  given in \eqref{Fj1:0},~\eqref{Fj:0}. Using the symbolic calculus Theorem \ref{theo:sc} one obtains without any difficulty that
\[
\lA F_j^1\rA_{L^2H^{s-\mez}}\les \lA u\rA_{L^2(I; H^s)}+\lA f\rA_{L^2(I; H^{s-\mez})}.
\]
For  $F_j^3$ we consider for example
\begin{align*}
 A_j&:=S_{{(j-3)}\mez}\gamma(x,D_x)\Delta_ju-S_{j-3}\gamma(x, D_x)\Delta_ju\\
&=\lp S_{{(j-3)}\mez}\gamma(x,D_x)\Delta_j u-\gamma(x, D_x)\Delta_j u\rp+\lp \gamma(x, D_x)\Delta_j u-S_{j-3}\gamma(x, D_x)\Delta_ju\rp\\
&=A_j^1+A_j^2.
\end{align*}
More generally, let $a\in \Gamma_\rho^m$ be homogeneous of degree $m$ in $\xi$.  Using the spherical harmonic decomposition we can assume $a(x, \xi)=b(x)c(\xi)$ with $b\in W^{\rho, \infty}$ and $c$  homogeneous of order $m$. Then $S_{\delta j}(a)(x, \xi)=S_{\delta j}(b)(x)c(\xi)$ and
\bq\label{Sja:a}
\lA (S_{\delta j}(a)(x,D_x)-a(x,D_x)) v\rA_{L^2}\le \lA S_{\delta j}(b) -b\rA_{L^\infty}\lA c(D_x)v\rA_{L^2}\les 2^{-\delta j\rho}\lA v\rA_{H^m}.
\eq
Since $\gamma\in \Gamma^{\tdm}_{2}$ is homogeneous of degree $\tdm$ in $\xi$, the $H^{s-\mez}$-norm of $A_j^1$  can be bounded by 
\[
2^{j(s-\mez)-\mez j 2+\tdm j}\lA \Delta_ju\rA_{L^2}\les \lA u\rA_{H^s}
\]
while 
\[
\lA A_j^2\rA_{H^{s-\mez}}\les 2^{j(s-\mez)-2j+ j}\lA \Delta_ju\rA_{H^\tdm}\les \lA u\rA_{H^{s-1}}.
\]
Similarly, we have 
\[
\lA F_j^3\rA_{L^2H^{s-\mez}}\les \lA u\rA_{L^2(I; H^s)}+\lA f\rA_{L^2(I; H^s)}.
\]
The estimate \eqref{Str:dyadic:0} then implies
\[
\lA \Delta_ju\rA_{L^{2}(I;W^{s-\frac{d}{2}+\mez-\eps,\infty})}\les_\eps \lA u\rA_{L^\infty(I; H^s)}. 
\]
Gluing these estimates together leads to
\bq\label{Str:u}
\begin{aligned}
\lA u\rA_{L^{2}(I;W^{s-\frac{d}{2}+\mez-2\eps,\infty})}&\le \sum_{j} 2^{-j\eps}\lA \Delta_ju\rA_{L^{2}(I;W^{s-\frac{d}{2}+\mez-\eps,\infty})}\\
&\les_\eps \lA u\rA_{L^\infty(I; H^s)}+\lA f\rA_{L^2(I; H^{s-\mez})}. 
\end{aligned}
\eq
Recall from Theorem \ref{theo:sym:0} that $u=T_p\eta+iT_q(\psi-T_B\eta)$. From \eqref{Str:u} one can use the symbolic calculus for paradifferential operators in Theorem \ref{theo:sc} to recover the corresponding estimates for $(\eta, \psi)$ (cf. \cite{ABZ1}, \cite{ABZ2}):
\[
\lA \eta\rA_{L^{2}(I;W^{s-\frac{d}{2}+1-2\eps,\infty})}+\lA \psi\rA_{L^{2}(I;W^{s-\frac{d}{2}+\mez-2\eps,\infty})}\les \cF_\eps\left(\lA(\eta, \psi)\rA_{L^\infty([0, T]; H^{s+\mez}\times H^s)}\right),
\]
where $\cF:\xR^+\to \xR^+$. The proof of  Theorem \ref{main:theo:0} is complete.
%%%%%%%%%%%%%%%%%%%%%%
%%%%%%%%%%%%%%%%%%%%%%%%%%%%%%%%%%%%%
\section{Proof of Theorem \ref{main:theo}}
We consider three parameters $\delta\in (0, 1),~r_0\in [0, 1],~r_1\in [0,\mez]$ which shall be determined later. Assume furthermore  that
\bq\label{defi:r0r1}
V\in L^\infty(I; W^{1+r_0,\infty}(\xR^d)),~\gamma(\cdot, \xi)\in L^\infty(I; W^{\mez+r_1,\infty}(\xR^d)).
\eq
\subsection{Littlewood-Paley reduction}
For every $j\geq 0$, the dyadic piece~$\Delta_ju$ solves 
\begin{equation}\label{eq:dyadic}
	\lp\partial_t+T_V\cdot\nabla +iT_\gamma\rp\Delta_ju=F_j,
\end{equation}
where
\begin{equation}\label{Fj1}
	F^1_j:=\Delta_jf +i\lb T_\gamma,\Delta_j\rb+\lb T_V,\Delta_j\rb\cdot\nabla u.
\end{equation}
In view of Proposition \ref{para:pseudo}, one has
\begin{equation} \label{eq:pseudo}
	\lp\partial_t+S_{j-3}(V)\cdot\nabla+iS_{j-3}(\gamma)(x, D_x) \rp\Delta_ju=F^2_j
\end{equation}
with
\bq\label{Fj2}
F^2_j:=F^1_j +R_j,
\eq
 $R_ju$ spectrally supported in an annulus $\{ c_1^{-1}2^{j-1}\le |\xi|\le c_12^{j+1}\}$ and
\bq\label{bound:Rj}
\lA R_j\rA_{H^s}\les \lA u\rA_{H^s}.
\eq
%Next, we smooth out the symbol by $\delta$ derivatives by  defining for all $(t,x,\xi)\in I\times\xR^d\times\xR^d$ and~$ j\geq 0$,
%\[
%S_{{(j-3)}\delta}(\gamma)(t,x,\xi)=\psi(2^{-{(j-3)}\delta}D_x)\gamma(t,x,\xi)
%\]
%and similarly for $S_{{(j-3)}\delta}(\omega)(t,x,\xi)$, $S_{{(j-3)}\delta}(V)(t,x)$.\\
Now, let $\mez<c_1<...<c_5$, $\Cr^k:=\{(2c_k)^{-1}\le |\xi|\le 2c_k\},~k=\overline{1,5}$ and 
\[
\tph\in C^\infty,~\supp\tph\subset \Cr^5,~\tph\equiv 1~\text{on}~\Cr^4.
\]
Equation~(\ref{eq:pseudo}) is then equivalent to
\bq\label{eq:reg}
	\lp\partial_t+ S_{{(j-3)}\delta}(V)\cdot\nabla+iS_{{(j-3)}\delta}(\gamma)(x, D_x)\tph(2^{-j}D_x)\rp\Delta_ju=F_j,
\eq
 with
\begin{multline}	\label{Fj}
	F_j=F_j^2+F_j^3:=F_j^2+i\lp S_{{(j-3)}\delta}\gamma(x,D_x)-S_{j-3}\gamma(x, D_x)\rp\Delta_ju\\ 
	      +\lp S_{{(j-3)}\delta}(V)-S_{j-3}(V)\rp\cdot\nabla \Delta_ju.
\end{multline}
Let us define the operator corresponding to the homogeneous problem of \eqref{eq:reg}
\bq\label{Lu}
\mathcal{L}_j:=\partial_t+ S_{{(j-3)}\delta}(V)\cdot\nabla+iS_{{(j-3)}\delta}(\gamma)(x, D_x)\tph(2^{-j}D_x).
\eq
To prove Strichartz estimate for $\Delta_ju$ as a solution to \eqref{eq:reg}, we shall first establish a ``pseudo dispersive estimate'' for $\mathcal{L}_j$. Set 
\[
h:=2^{-j},~\widetilde h:=h^{\mez}.
\]
%%%%%%%%%%%%%55
\subsection{Straightening the vector field}\label{Straighten}
Following \cite{ABZ4} we straighten the vector field $\partial_t+S_{(j-3)\delta}\cdot \nabla$ by considering the system 
\begin{equation}\label{eqdif}
\left \{
\begin{aligned}
&\dot{X}_k(s) =  S_{(j-3)\delta}(V_k)(s,X(s)), \quad 1 \leq k \leq d,\quad X=(X_1,\ldots,X_d) \\
&X_k(0) =x_k.
\end{aligned}
\right.
\end{equation}
Since $V\in L^\infty([0, T]; L^\infty_x)$, system \eqref{eqdif} has a unique solution on $I=[0, T]$, which shall be denoted for simplicity by $X(s, x; h)$, or even $X(s, x)$. Estimates for the flow $s\mapsto X(s, \cdot)$ is given in the next proposition.
\begin{prop}\label{est:X}
For fixed $(s,h)$ the map $x \mapsto X(s, x; h)$ belongs to $C^\infty (\xR^d, \xR^d).$ Moreover,  for all $(s,h) \in I\times (0, 1]$  we have
\begin{align}
& \Vert (\partial_x X) (s, \cdot; h) -\text{Id} \Vert_{L^\infty(\xR^d)} 
\le \cF\big(\lA V\rA_{L^\infty([0, T]; W^{1,\infty}}\big)|s|,\quad \label{X1}\\
& \Vert (\partial_x^\alpha X) (s, \cdot; h)  \Vert_{L^\infty(\xR^d)} 
\le \cF_\alpha\big(\lA V\rA_{L^\infty([0, T]; W^{1+r_0,\infty}}\big)h^{-\delta(|\alpha|-(1+r_0))}|s|, \quad    \vert \alpha \vert \ge 2  \label{X2},
\end{align}
where $\cF$, $\cF_\alpha:\xR^+\to \xR^+$.
\end{prop}
\begin{proof}
Here we follow the poof Proposition $4.10$, \cite{ABZ4}. The improvement  is due to the following estimate (by applying Lemma \ref{Bernstein})
\bq\label{est:flow:improve}
\lA \partial_x^\beta S_{j\delta}(V)(s)\rA_{L^\infty(\xR^d)}\le C_\beta h^{-\delta(|\beta|-1-r_0)}\lA V(s)\rA_{W^{1+r_0,\infty}}\quad\forall |\beta|\ge 2.
\eq
$(i)$ To prove \eqref{X1} we differentiate \label{eqdiff} with respect to $x_l$ to obtain
\begin{equation*}
\left\{
\begin{aligned}
&\dot{\frac{\partial X_k}{\partial x_l}}(s) = 
 \sum_{q = 1}^d S_{j\delta}\Big (\frac{\partial V_k}{\partial x_q}\Big)(s,X(s))\frac{\partial X_q}{\partial x_l}(s)\\
&\frac{\partial X_k}{\partial x_l}(0) = \delta_{kl}
\end{aligned}
\right. 
\end{equation*}
from which we deduce 
\begin{equation}\label{dX}
\frac{\partial X_k}{\partial x_l}(s) = \delta_{kl} 
+ \int_0^s  \sum_{q = 1}^d S_{j\delta}
\Big (\frac{\partial V_k}{\partial x_q}\Big)(\sigma,X(\sigma))\frac{\partial X_q}{\partial x_l}(\sigma) \, d \sigma.
\end{equation}
Setting 
$\vert \nabla X\vert = \sum_{k,l =1}^d \vert \frac{\partial X_k}{\partial x_l} \vert$ 
we obtain from \eqref{dX}
$$
\vert \nabla X(s) \vert \leq C_d + \int_0^s \vert \nabla V(\sigma, X(\sigma)) \vert  \, \vert \nabla X(\sigma) \vert \, d \sigma.
$$
The Gronwall inequality implies 
\begin{equation}\label{Gron}
 \vert \nabla X(s) \vert \leq \mathcal{F} (\Vert V \Vert_{W^{1,\infty}}) \quad \forall s \in I. 
 \end{equation}
Coming back to \eqref{dX} and using \eqref{Gron} lead to
$$
\la \frac{\partial X }{\partial x}(s) - Id \ra \leq  \mathcal{F} (\Vert V \Vert_{W^{1,\infty}}) \int_0^s \Vert \nabla V(\sigma, \cdot)  \Vert_{L^\infty(\xR^d)} 
\, d\sigma \leq \mathcal{F}_1 (\Vert V \Vert_{W^{1,\infty}})|s|.
$$
$(ii)$ We shall prove  \eqref{X2} for $|\alpha|=2$ first and then prove by induction on $\vert \alpha \vert$ that the estimates
$$
\Vert(\partial_x^\alpha X)(s;\cdot, h)\Vert_{L^\infty(\xR^d)} \leq \mathcal{F}_\alpha(\Vert V \Vert_{W^{1+r_0,\infty}}) h^{-\delta(\vert \alpha \vert -1-r_0)}
$$
for  $2 \leq \vert \alpha \vert \leq k$ implies \eqref{X2} for  $\vert \alpha \vert = k+1.$ \\
 Differentiating  $\vert \alpha \vert$ times ($|\alpha|\ge 2$) the system \eqref{eqdif} and using the  Faa-di-Bruno formula  we obtain
\begin{equation}\label{rec1}
 \frac{d}{ds}\big(\partial_x^\alpha X\big)(s) = S_{j\delta}(\nabla V)(s,X(s)) \partial_x^\alpha X +(1) 
 \end{equation}
where the term $(1)$ is a finite linear combination of terms of the form
$$
A_\beta (s,x) = \partial_x^\beta \big(S_{j\delta}(V)\big)(s,X(s)) \prod_{i=1}^q \big(\partial_x^{L_i} X(s)\big)^{K_i}
$$
where 
$$
2 \leq \vert \beta \vert \leq \vert \alpha \vert, \quad   \vert L_i \vert, \vert K_i\vert \ge 1, \quad \sum_{i=1}^q \vert K_i \vert L_i = \alpha, \quad \sum_{i=1}^q K_i = \beta.
$$
1. When $|\alpha|=2$, we have
\[
A_\beta (s,x) = \partial_x^\beta \big(S_{j\delta}(V)\big)(s,X(s)) \prod_{i=1}^q \big(\partial_x^{L_i} X(s)\big)^{K_i}
\]
with $|L_i|=1$ and $|\beta|=|\alpha|=2$. It then follows from $(i)$ that 
\[
\vert  \prod_{i=1}^q \big(\partial_x^{L_i} X(s)\big)^{K_i}\vert \le \mathcal{F} (\Vert V \Vert_{W^{1,\infty}}).
\]
On the other hand, we have by \eqref{est:flow:improve} 
\[
\lA \partial_x^\beta S_{j\delta}(V)(s)\rA_{L^\infty(\xR^d)}\le Ch^{-\delta(|\alpha|-1-r_0)}\lA V(s)\rA_{W^{1+r_0,\infty}}.
\]
Consequently, it holds that 
\[
\Vert (1)\Vert_{L^\infty(\xR^d)}\le h^{-\delta(|\alpha|-1-r_0)}\mathcal{F} (\Vert V \Vert_{W^{1+r_0,\infty}}),
\]
from which we obtain \eqref{X2} for $|\alpha|=2$ by using \eqref{rec1} and Gronwall's inequality.\\
2. Assuming now that \eqref{X2} holds with $2\le |\alpha|\ge  k$, we shall prove it for $|\alpha|=k+1$. Indeed,  from \eqref{X1} and the induction hypothesis it holds for any $1\le |\nu|\le k$ that 
\[
\Vert (\partial_x^\nu X) (s, \cdot; h)  \Vert_{L^\infty(\xR^d)} 
\le_\alpha \cF\big(\lA V\rA_{L^\infty([0, T]; W^{1+r_0,\infty}}\big)h^{-\delta(|\nu|-1)}|s|.
\]
Because $\vert \beta\vert \ge 2$ and $\vert L_i \vert \ge 1$, using \eqref{est:flow:improve} and the preceding estimate we obtain
\begin{equation*}
\begin{aligned}
\Vert A_\beta (s,\cdot) \Vert_{L^\infty(\xR^d)} 
&\leq \big\Vert \partial_x^\beta \big(S_{j\delta}(V)\big)(s,\cdot)\big\Vert_{L^\infty(\xR^d)} \prod_{i=1}^q 
\Big\Vert \partial_x^{L_i} X(s, \cdot)\Big\Vert _{L^\infty(\xR^d)}^{\vert K_i \vert}\\[1ex]
&\leq C h^{-\delta (\vert \beta \vert -1-r_0)}\Vert V(s,\cdot) \Vert_{W^{1+r_0, \infty}} 
h^{-\delta \sum_{i=1}^q \vert K_i\vert(\vert L_i \vert -1)}\mathcal{F} (\Vert V \Vert_{W^{1+r_0,\infty}}) \\[1ex]
&\leq h^{-\delta(\vert \alpha \vert -1-r_0)}\mathcal{F} (\Vert V \Vert_{W^{1+r_0,\infty}})\Vert V(s,\cdot) \Vert_{W^{1+r_0, \infty}}.
\end{aligned}
\end{equation*}
As before, we conclude by \eqref{rec1} and Gronwall's inequality.
%\begin{equation}
%\begin{aligned}
%\vert \partial_x^\alpha X(s) \vert &\leq   \mathcal{F} (\Vert V \Vert_{W^{1+r_0,\infty}})h^{-\delta(\vert \alpha \vert -1-r_0)} 
%\int_0^s\Vert V(\sigma,\cdot) \Vert_{W^{1+r_0, \infty}(\xR^d)} \,d\sigma\\
%&\quad + C \int_0^s\Vert V(\sigma,\cdot) \Vert_{W^{1+r_0, \infty} (\xR^d)}\vert \partial_x^\alpha X(\sigma) \vert \, d\sigma.
%\end{aligned}
%\end{equation}

\end{proof}
In view of \eqref{X1} the mapping $x\mapsto X(s, x; h)$ is a $C^\infty$-diffeomorphism when $0\le s\le T_0$ small enough. This restriction is harmless for one can iterate the final estimate over time intervals of length $T_0$ which depends only on $\lA V\rA_{L^\infty([0, T]; W^{1,\infty}})$.\\
Now, in \eqref{Lu} we first  make the change of spatial variables 
\bq
v_h(t, y)=u_j(t, X(t, y; h))%,\quad G_h(t, y)=F_j(t, X(t, y; h))
\eq
so that 
\bq\label{dt:vh}
\lp\partial_t+ S_{{(j-3)}\delta}(V)\cdot\nabla\rp u_j(t, X(t, y; h)=\partial_tv_h(t, y).
\eq
Denoting 
\bq\label{def:qh}
q_h(x, \xi):=S_{{(j-3)}\delta}(\gamma)(x, \xi)\tph(h\xi),%, \quad Q_h:=\Op(q_h).
\eq
let us compute  this dispersive term after the above change of variables. To this end, set
\begin{equation}\label{MHJ}
\begin{aligned}
H(y,y') &= \int_0^1 \frac{\partial X}{\partial x}(\lambda y +(1-\lambda) y') \, d\lambda, 
\quad M(y,y') =\big(\, {}\!^tH(y,y') \big)^{-1} \\
M_0(y)  &= \Big(\, {}\!^t  \Big(\frac{\partial X}{\partial x}(y)\Big)  \Big)^{-1}, 
\quad  J(y,y') = \Big\vert \det \Big(\frac{\partial X}{\partial x}(y')\Big)\Big\vert \vert \det M(y,y') \vert.
\end{aligned}
\end{equation}
Then,
\[
(\Op(q_h)u_j)\circ X(y)=(2\pi)^{-d}\iint e^{i(X(y)-x')\cdot \eta } q_h(X(y), \eta) u_j(x') dx' d\eta.
\]
Now, we make two changes of variables $x'=X(y')$ and $\eta=M(y, y')\zeta$ to obtain
\[
(\Op(q_h)u_j)\circ X(y)= (2\pi)^{-d} \iint e^{i (y-y')\cdot \zeta } q_h\big(X(y), M(y,y')\zeta\big) J(y,y') v_h(y') \,dy' \,d\zeta.
\]
Observe that the above pseudo-differential operator is still of order $\mez$. To change its order to $1$, we make another change of spatial variables
\bq
y=h^\mez z=\th z,\quad y'=\th z',\quad ~w_h(z')=v_h(\th z'),\quad \xi=\th \zeta
\eq
so that
\begin{multline}\label{A1}
(\Op(q_h)u_j)\circ X(y)  = (2\pi)^{-d} \iint e^{i (z-z')\cdot \xi } 
 q_h\big(X(\th z), M\big(\th z,\th z'\big)\th ^{-1}\xi\big)\times \\
 \times J\big(\th z,\th z'\big) w_h(z') \,dz' \,d\xi.
 \end{multline}
Summing up, with 
\bq\label{def:ph}
p_h(z, z', \xi):= q_h\big(X(\th z), M\big(\th z,\th z'\big)\th ^{-1}\xi\big)
 J\big(\th z,\th z'\big), \quad w_h(t, z)=u_j(t, X(t, \th z))
\eq
it holds that
\[
(\Op(q_h)u_j)\circ X(\th z)  =\Op(p_h)w_h(z),
\]
which, combing with \eqref{dt:vh} and \eqref{Lu}, yields
\bq\label{eq:order1}
\lp\mathcal{L}_ju_j\rp(t, X(t, \th z))=\lp\partial_t +i\Op(p_h)\rp w_h(t, z),\quad  \quad w_h(t, z)=u_j(t, X(t, \th z)).%=H(t, z):=G(t, \th z),
\eq
%where we recall that 
%\bq\label{def:wh}
%x=X(t, \th z),\quad w_h(t, z)=(\Delta_ju)\circ (X(t, \th z)).
%\eq
We have transformed the operator $\mathcal{L}_j$ of order $\mez$ into the right-hand side of \eqref{eq:order1} which has order $1$.
%%%%%%%%%%%%%%%%%%%%%%%%55
\subsection{Approximation of the symbol $p_h$}\label{Approximation}
Observe that $p_h$  depends on $(z, z', \xi)$ which is not in the standard form to use the phase  space transform in \cite{KT}. We will approximate $p_h$ by some symbol depending only on $(z, \xi)$. A general result can be found in Proposition $0.3A$ \cite{Taylor}. However, we will inspect more carefully the smoothness of $p_h$ to obtain better estimates for the error. To do this, we write as in \eqref{toW}-\eqref{toW:remainder}
\begin{align*}
p_h(z, z', \xi)&=p_h(z, z, \xi)+\int_0^1\partial_{z'}p_h(z, z+s(z'-z), \xi)ds (z'-z)\\
&:=p_h^0( z,\xi)+r_h^0(z, z' ,\xi)(z'-z),
\end{align*}
where 
\bq\label{def:p0}
p_h^0( z, \xi)=p_h(z,z,\xi)= q_h\big(X(\th z), M_0(\th z)\th ^{-1}\xi\big).
 \eq
On the other hand,
\[
\Op(r_h^0.(z'-z))w(z)=-i\Op(r_h) w(z)
\]
with 
\[
r_h(z, z', \xi)=\int_0^1\partial_\xi\partial_{z'} p_h(z, z+s(z'-z), \xi)ds.
\]
 To simplify notations, we denote $[z,z']_s=z+s(z'-z)$ so that 
\bq\label{def:rh}
 r_h(z, z', \xi)=\int_0^1\partial_\xi \partial_{z'}q_h\big(X(\th z), M\big(\th z,\th [z,z']_s\big)\th ^{-1}\xi\big) J\big(\th z,\th [z,z']_s\big)ds.
\eq
In conclusion, $\Op(p_h)=\Op(p_h^0)-i\Op(r_h)$.
\subsubsection{The symbol $p^0_h$}
First,  Lemma \ref{est:X} implies directly estimates for $M$ and $J$.
\begin{lemm}\label{est:MJ} It holds for all $(\alpha, \alpha')\in (\xN^d)^2$ that
\[
\vert \partial_z^\alpha\partial_{z'}^{\alpha'}M(z, z')\vert+\vert \partial_z^\alpha\partial_{z'}^{\alpha'}J(z, z')\vert \les_{\alpha, \alpha'}
\left\{
\begin{aligned}
1  &\quad\text{if}~|\alpha|+|\alpha'|= 0,\\
\th^{-2\delta(|\alpha|+|\alpha'|-r_0)}& \quad\text{if}~|\alpha|+|\alpha'|\ge 1.
\end{aligned}
\right.
\]
\end{lemm}
On the other hand, by Bernstein's inequalities and the fact that on the support of $\tph(h\xi)$,~$|\xi|\sim \th^{-2}$, we can estimate the derivatives of $q_h$, given by \eqref{def:qh}, as follows.
\begin{lemm}\label{est:qh}
We have for all $(\alpha, \beta)\in (\xN^d)^2$
\[
\la \partial_x^\alpha\partial_\xi^\beta q_h(x, \xi)\ra \les_{\alpha, \beta}
\left\{
\begin{aligned} 
\th^{-1+2|\beta|},\quad\text{if}~|\alpha|=0,\\
\th^{-1-2\delta(|\alpha|-(\mez+r_1))+2|\beta|},\quad\text{if}~|\alpha|\ge 1.
\end{aligned}
\right. 
\]
\end{lemm}
We now study the regularity of the symbol $p_h^0$.
\begin{prop}\label{est:p0}
Choosing $r_0,~r_1$ satisfying
\bq\label{r0r1}
2\delta(1-r_0)\le 1,~2\delta(2-r_0)\le 2,~2\delta(\mez-r_1)\le 1,~2\delta(\tdm-r_1)\le 2
\eq
then the symbol $p_h^0$ verifies \\
$(i)$ for all $(\alpha, \beta)\in \xN^d,~|\alpha|\le 2$
\bq\label{p0:C2}
\la \partial_z^\alpha \partial_\xi^\beta p_h^0(z, \xi)\ra\les_{\alpha,\beta} \th^{-1+|\beta|}\mathbbm{1}_{\th |\xi|\sim 1}(\xi),
\eq
$(ii)$ for all $(\alpha, \beta)\in \xN^d,~|\alpha|\ge 3$
\bq\label{p0:high}
\la \partial_z^\alpha \partial_\xi^\beta p_h^0(z, \xi)\ra\les_{\alpha,\beta} \th^{-1-(2\delta-1)(|\alpha|-2)+|\beta|}\mathbbm{1}_{\th |\xi|\sim 1}(\xi). 
\eq
\end{prop}
\begin{proof}
 To simplify notations, we shall denote in this proof $q\equiv q_h$.\\
$(i)$ Observe that \eqref{p0:C2} is trivial when $\alpha=0$. The argument below is independent of the dimension, however let us further simplify the notations by writing as if $d=1$. For $|\alpha|=1$, we compute
\bq\label{dxp}
\begin{aligned}
\partial_z^\alpha p^0(z, \xi)&=q_x\big(X(\th z), M_0(\th z)\th ^{-1}\xi\big)\th X'(\th z)\\
& \quad+q_\xi\big(X(\th z), M_0(\th z)\th ^{-1}\xi\big)(\th^{-1}\xi)\th M'_0(\th z).
\end{aligned}
\eq
 For $|\alpha|=2$, we have %we write with $(\cdot)=\big(X(\th z), M_0(\th z)\th ^{-1}\xi\big)$,
\bq\label{d2xp} 
\begin{aligned}
\partial_z^\alpha p_h^0(z, \xi)&=q_{xx}(\cdots)\th^2(X')^2+2q_{x\xi}(\cdots)X'M'_0\th \xi+q_{\xi\xi}(\cdots)(M_0')^2\xi^2,\\
&\quad +q_x(\cdots)\th^2X''+q_\xi(\cdots)(\th \xi)M_0''.
\end{aligned}
\eq
Remark that on the support of $p^0(z, \xi)$,~$\th |\xi|\sim 1$, using Proposition \ref{est:X}, Lemmas \ref{est:MJ} and \ref{est:qh} one deduces easily that 
\begin{align*}
&\la \partial^\alpha_z p^0(z, \xi)\ra \les_{\alpha}  \th^{-1-2\delta(\mez-r_1)+1}+\th^{-1+2-1-2\delta(1-r_0)},\quad|\alpha|=1,\\
&\la \partial^\alpha_z p^0(z, \xi)\ra \les_{\alpha} \th^{-1-2\delta(\tdm-r_1)+2}+\th^{-1-2\delta(\mez-r_1)+2-2\delta(1-r_0)}+\th^{-1+4-4\delta(1-r_0)-2}\\
&\quad +\th^{-1-2\delta(\mez-r_1)+2-2\delta(1-r_0)}+\th^{-1+2-2\delta(2-r_0)},\quad|\alpha|=2.
\end{align*}
Under conditions \eqref{r0r1}, we get
\[
\la \partial^\alpha_z p^0(z, \xi)\ra \les_{\alpha} \th^{-1}\mathbbm{1}_{\th |\xi|\sim 1}(\xi),\quad|\alpha|\le 2. 
\]
To obtain  $(i)$ it remains to estimate $\partial_\xi^\beta(\partial_z^\alpha p_0^h)$ for $|\alpha|\le 2$ and $\beta\in \xN^d$. From the explicit expressions \eqref{dxp}, \eqref{d2xp} of $\partial_z^\alpha p^0_h$, we see that there are two possibilities when differentiating once in $\xi$. One possibility is that the derivative falls down to $q$. This makes  appear the factor $M_0(\th z)\th^{-1}$ while we gain $\th^2$ when differentiating $q$ in $\xi$ (by Lemma \ref{est:qh}), we thus gain $\th$. Another possibility is that the derivative falls down to $\xi^\nu,~\nu=1, 2$, which  results in $\nu\xi^{\nu-1}$. Since $\xi\sim \th^{-1}$ on the support of $p^0_h$ one deduces that $\xi^{\nu-1}\sim \xi^\nu \th$, which means that we still gain  $\th$. Therefore, in both cases we gain $\th$ when differentiating once in $\xi$ and thus \eqref{p0:C2} follows.\\
$(ii)$ As just explained above, it suffices to prove \eqref{p0:high} with $\beta=0$. From the formula  \eqref{d2xp}, the proof of \eqref{p0:high} reduces to showing for $|\alpha|\ge 0$
\[
\la \partial_z^\alpha A_j(z, \xi, h)\ra\les_{\alpha,\beta} \th^{-1-(2\delta-1)|\alpha|}\mathbbm{1}_{\th |\xi|\sim 1}(\xi),~j=\overline{1,5}
\]
with 
\[
\left\{
\begin{aligned}
&A_1=q_{x\xi}\big(X(\th z), M_0(\th z)\th ^{-1}\xi\big)\th^{-1},\\
&A_2=q_{xx}\big(X(\th z), M_0(\th z)\th ^{-1}\xi\big)\th^2,\\
&A_3=q_{\xi\xi}\big(X(\th z), M_0(\th z)\th ^{-1}\xi\big)\th^{-4},\\
&A_4=q_x\big(X(\th z), M_0(\th z)\th ^{-1}\xi\big)\th,\\
&A_5=q_{\xi}\big(X(\th z), M_0(\th z)\th ^{-1}\xi\big)\th^{-2}
\end{aligned}
\right.
\]
and 
\[
\la \partial_z^\alpha B_j(z, h)\ra\les_{\alpha,\beta} \th^{-(2\delta-1)|\alpha|},~j=\overline{1,4}
\]
with
\[
\left\{
\begin{aligned}
&B_1=X'(\th z),\\
&B_2=\th M'_0(\th z),\\
&B_3=\th X''(\th z),\\
&B_4=\th^2M''(\th z).
\end{aligned}
\right.
\]
1. $B_j$. By Lemma \ref{X1},
\[
\vert \partial^\alpha_zB_1\vert=\vert \partial^\alpha_zX'(\th z)\vert =\th^{|\alpha|}\vert (\partial^{\alpha+1}_xX)(\th z)\vert\les \th^{|\alpha|-2\delta|\alpha|}\les\th^{-1-(2\delta-1)|\alpha|}.
\]
On the other hand, \eqref{X2} and the condition $2\delta(1-r_0)\le 1$ imply
\[
\vert \partial^\alpha_zB_3\vert=\th \vert \partial^\alpha_zX''(\th z)\vert =\th^{1+|\alpha|}\vert (\partial^{\alpha+2}_xX)(\th z)\vert\les \th^{1+|\alpha|-2\delta(|\alpha|+1-r_0)}\les\th^{-(2\delta-1)|\alpha|}.
\]
Remark that $M_0'(\th z)$ is as smooth as $X''(\th z)$, the preceding estimate also holds for $B_2$. For $B_4$, we  apply \eqref{X2} and use the condition $2\delta(2-r_0)\le 2$ to estimate
\[
\vert \partial^\alpha_zB_4\vert=\th^2 \vert \partial^\alpha_zM_0''(\th z)\vert =\th^{2+|\alpha|}\vert (\partial^{\alpha+2}_xM_0)(\th z)\vert\les \th^{|\alpha|+2-2\delta(|\alpha|+2-r_0)}\les\th^{-(2\delta-1)|\alpha|}.
\]
2. $A_1$. For $\alpha=0$, Lemma \ref{est:qh} gives
\[
|A_1|\les \th^{-1-2\delta(\mez-r_1)+2-1}\les \th^{-1}
\]
since $2\delta(\mez-r_1)\le 1$. Considering now $|\alpha|\ge 1$,  using the Faa-di-Bruno formula we see that $\partial_z^\alpha A_1$ is a linear combination of terms of the form
\[
C_1=\th^{|\alpha|-1}\big(\partial_x^{a+1}\partial_\xi^{b+1}q\big)(\cdots)\prod_{j=1}^r\big((\partial_x^{L_j}X)(\th z)\big)^{P_j}\big((\partial^{L_j}_xM_0)(\th z)\th^{-1}\xi\big)^{Q_j}
\]
where $1\le |a|+|b|\le |\alpha|$, $|L_j|\ge 1$ $\forall j=\overline{1,r}$ and 
\[
\sum_{j=1}^r P_j=a,~\sum_{j=1}^rQ_j=b,~\sum_{j=1}^r(|P_j|+|Q_j|)L_j=\alpha.
\]
According to Lemma \ref{est:qh},
\[
\vert \big(\partial_x^{a+1}\partial_\xi^{b+1}q\big)(\cdots)\vert\les \th^{-1-2\delta(|a|+\mez-r_1)+2(|b|+1)}.
\]
On the other hand, since $|L_j|\ge 1$ Lemmas \ref{est:qh}, \ref{est:MJ} allow us to estimate the product appearing in $C_1$ as follows
\begin{align*}
\vert \prod_{j=1}^r\vert\les \th^K,~K&= \sum_{j=1}^r\Big(-2\delta(|L_j|-1)|P_j|-2\delta(|L_j|-r_0)|P_j|\Big)-2\sum_{j=1}^r|Q_j|\\
&=-2\delta|\alpha|+2\delta |a|+2\delta r_0|b|-2|b|.
\end{align*}
Therefore, $|C_1|\les \th^L$ with 
\begin{align*}
L&=|\alpha|-1-1-2\delta(|a|+\mez-r_1)+2(|b|+1)-2\delta|\alpha|+2\delta |a|+2\delta r_0|b|-2|b|\\
&\ge -1-(2\delta-1)|\alpha|+1-2\delta(\mez-r_1) 
\\
& \ge-1-(2\delta-1)|\alpha|,
\end{align*}
where we have used again the condition that $2\delta(\mez-r_1) \le 1$. The proof for $A_1$ is complete.\\
3. $A_2,~A_3,~A_4,~A_5$. The estimate  for these terms can be derived along the same lines as for $A_3$, where one need to make use of the condition $2\delta (\tdm-r_1)\le 2$ for $A_2$ and the condition $2\delta (\mez-r_1)\le 1$ for $A_4$.
\end{proof}

From now on, we always assume condition \eqref{r0r1} for $r_0$ and $r_1$.
\subsubsection{The symbol $r_h$}
The next lemma provides the order of $r_h$ and shows that it decays in $\xi$ faster than in $(z, z')$ which shall be important in our \enquote{pseudo-dispersive estimates} in paragraph \ref{p:pde}.
\begin{lemm}\label{est:rh}
For all $(\alpha, \alpha', \xi)\in (\xN^d)^3$, we have
\[
\la \partial_z^\alpha \partial_{z'}^{\alpha'} \partial_\xi^\beta r_h(z, z', \xi)\ra \les_{\alpha, \alpha',\beta} \th^{1-2\delta(1-r_0)-(2\delta-1)(|\alpha|+|\alpha'|)+|\beta|}\mathbbm{1}_{\{\th |\xi|\sim 1\}}(\xi).
\]
Consequently, $r_h\in S^{-1+2\delta(1-r_0)}_{1, (2\delta-1),(2\delta-1)}$.
\end{lemm}
\begin{proof}
Recall the definition \eqref{def:rh} of $r_h$. On the support of this symbol, $|\xi|\sim \th^{-1}$. In this proof, all the estimates are uniform in $s\in [0, 1]$. It follows from Lemma \ref{est:MJ} that 
\[
\forall (\alpha, \alpha')\in (\xN^d)^2, \quad\la \partial_z^\alpha \partial_{z'}^{\alpha'} J\big(\th z,\th [z,z']_s\big)\ra \les_{\alpha, \alpha'} \th^{-(2\delta-1)(|\alpha|+|\alpha'|)}.
\]
Next, setting
\[
\widetilde q_h(x, \xi)=S_{{(j-3)}\delta}(\gamma)(x, \xi)\tph(\xi)
\]
we see that 
\[
q_h\big(X(\th z), M\big(\th z,\th [z,z']_s\big)\th ^{-1}\xi\big)
=\th^{-1}\widetilde q_h\big(X(\th z), M\big(\th z,\th [z,z']_s\big)\th \xi\big).
\]
%\begin{align*}
%g^1(z,z',\xi)&=S_{{(j-3)}\delta}(\gamma)\big(X(\th z), M(\th z, \th [z,z']_t)\xi\big)\tph\big(M(\th z, \th [z,z']_t)\xi\big),\\ g(z, z',\xi)&=S_{{(j-3)}\delta}(\gamma)\big(X(\th z), M(\th z, \th[z, z']_t)\th^{-1}\xi\big)\tph\big(M(\th z, \th [z,z']_t)\th \xi\big)
%\end{align*}
%we see  $g(z, z', \xi)=\th^{-1}g^1(z, z', \th\xi)$. 
The proof of this lemma then boils down to showing for all $(\alpha, \alpha', \beta)\in (\xN^d)^3$, 
\bq\label{est:r:1}
%\forall (\alpha, \alpha', \xi)\in (\xN^d)^3:\quad 
\la  \partial_\xi^\beta \partial_z^\alpha \partial_{z'}^{\alpha'}\partial_\xi \partial_{z'}\widetilde q_h(X( z), M( z, [z,z']_s)\xi) \ra 
\les_{\alpha, \alpha',\beta} \th^{-2\delta(|\alpha|+|\alpha'|)-2\delta(1-r_0)}.
\eq
We compute 
\[
\Xi:=\partial_\xi \partial_{z'}\widetilde q(X(z), M( z, [z,z']_s)\xi)=s\widetilde q_{\xi\xi}(\cdots) M_{z'}\xi M+s\widetilde q_\xi(\cdots) M_{z'}
\]
which is bounded by $\th^{-2\delta(1-r_0)}$ in view of Lemma \ref{est:MJ} and the fact that $|\xi|\sim 1$ on the support of $\widetilde q$. For the same reason we see that taking $\xi$-derivatives of $\Xi$ is harmless (notice that $M$ is bounded), so we only need to prove \eqref{est:r:1} for $|\beta|=0$. Indeed,  by Lemma \ref{est:MJ},
\bq\label{est:MM'}
\la  \lp\partial_z^\alpha \partial_{z'}^{\alpha'}M\rp (\cdot)\ra+\la  \lp\partial_z^\alpha \partial_{z'}^{\alpha'}M_{z'}\rp (\cdot)\ra \les \th^{-2\delta(|\alpha|+|\alpha'|)-2\delta(1-r_0)}
\eq
 On the other hand, using the Faa-di-Bruno formula (as in the proof of Proposition \ref{est:p0}) we can prove that
\[
\la \partial_z^\alpha \partial_{z'}^{\alpha'}(\partial^\gamma_\xi \widetilde q)\big(X(z), M( z, [z,z']_s)\xi\big)\ra\les \th^{-2\delta(|\alpha|+|\alpha'|)},
\]
from which we conclude the proof.
\end{proof}
In view of equation \eqref{eq:order1} we have proved  that 
%The function $w_h(t, z)$ defined by \eqref{def:wh} solves the following equation
\bq\label{eq:wh}
\lp\mathcal{L}_ju_j\rp(t, X(t, \th z)=\lp\partial_t +i\Op(p^0_h)\rp w_h(t, z)-i\Op(r_h)w_h(t,z)
\eq
via the relation $w_h(t, z)=u_j(t, X(t, \th z))$.\\

%\begin{rema}
%This result can also be derived by applying Lemma \ref{toW:improve} and the fact that $p^0_h\in \ld S^2_\ld$ (see the proof of  below).
%\end{rema}
%%%%%%%%%%%%%%%%
\subsection{A \enquote{pseudo dispersive estimate} for $\mathcal{L}_j$}\label{p:pde}
In this step, we shall show that Theorem \ref{theo:KT} can be applied to the evolution operator
\[
L_h:=D_t+\Op^w(p^0_h).
\]
Henceforth, we set
\[
\delta=\frac{3}{4},\quad \ld=\th^{-1}.
\]
Proposition \ref{est:p0} shows that $p^0_h$ belongs to $\ld S^2_\ld$. Using Lemma \ref{toW:improve} to replace $\Op(p_h^0)$ in \eqref{eq:wh} with $\Op^w(p_h)$ we have
\bq\label{p:pw}
\Op^w(p_h^0)=\Op(p_h^0)+\Op(r'_h)
\eq
with $r'_h\in S^0_{1, \mez, \mez}$. On the other hand, since $2\delta(1-r_0)\le 1$, Lemma \ref{est:rh} combining with \eqref{eq:wh}, \eqref{p:pw} leads to the following.
\begin{prop} There holds for some symbol $r^1_h\in S^0_{1, \mez, \mez}$ that
\bq\label{eq:whw}
\frac{1}{i}\lp\mathcal{L}_ju_j\rp(t, X(t, \th z)=\lp D_t+\Op^w(p^0_h)\rp w_h(t, z)+\Op(r^1_h) w_h(t,z).%=:\widetilde L_hw_h(t,z).
\eq
\end{prop}
Next, we recall the following proposition which says that the  the characteristic set of $p^0_h$ has $d$ nonvanishing principal curvatures. 
\begin{prop}[\protect{\cite[Proposition 4.16]{ABZ4}}]\label{Hessp0}
Let $\Cr$ be a fixed annulus in $\xR^d$. For any $0<\delta<1$ there exist $m_0>0,~h_0>0$   such that
\[
		      	\sup_{(t, x, \xi, h)\in I\times \xR^d\times (\Cr\times(0, h_0]}\la\det\partial_\xi^2S_{\delta j}(\gamma)(t,x,\xi)\ra\geq m_0.	
\]
\end{prop}
Let $\mathcal{S}_j$ and $S_h$ denote the propagator of $\mathcal{L}_j$ and $L_h$, respectively. We are now in position to apply Theorem \ref{theo:KT} to derive dispersive estimates for $S_h$. 
\begin{prop}\label{disp:w}  For any symbol $\chi\in S^0_\ld$ satisfying for all $z\in \xR^d$ $\supp\chi(z,\cdot)\subset \ld\Cr^2$, we have
 \bq\label{disp:w}
		\lA S_h(t, t_0)\big( \chi(z, D_z)f\big)\rA_{L^\infty}\les  \th^{-\frac{d}{2}}
		  \la t-t_0\ra^{-\frac{d}{2}}\lA f\rA_{L^1}
\eq
for all~$t,~ t_0\in [0, 1]$ and~$0<\th\leq \th_0$. \\
If in addition, $\chi(z, D_z):L^2\to L^2$ then for any $r\in [2, \infty]$ there holds by interpolation
\bq\label{disp:w:inter}
		\lA S_h(t, t_0)\big( \chi(z, D_z)f\big)\rA_{L^r}\les  \th^{-d(\mez-\frac{1}{r})}
		  \la t-t_0\ra^{-d(\mez-\frac{1}{r})}\lA f\rA_{L^{r'}},
\eq
where $r'$ is the conjugate exponent of $r$, {\it i.e.}, $\frac{1}{r}+\frac{1}{r'}=1$.
\end{prop}
\begin{proof}
We have seen that $p^0_h\in \ld S^2_\ld$. On the other hand,  $\tph\equiv 1$ in $\Cr^4$. Proposition \ref{Hessp0} then gives 
\[
		 	\sup_{(t, x, \xi, h)\in I\times \xR^d\times \Cr^4\times(0, h_0]}\la\det\lp \partial_\xi^2S_{\delta j}(\gamma)(t,x,\xi)\tph(\xi)\rp\ra\gs 1.	
\]
Remark that  \eqref{X1} implies $|M_0(y)|\ge c_0$ for all $y\in \xR^d$ (by choosing $T$ small enough as explained after Proposition \ref{est:X}). Consequently,
\[
		 	\sup_{(t, x, \xi, h)\in I\times \xR^d\times( \ld\Cr^3)\times(0, h_0]}\la\det\partial_\xi^2p_h^0\ra\gs \ld^{-d}	
\]
if $c_3<c_4$ is chosen appropriately. In other words, condition $(\bf A)$ in Theorem \ref{theo:KT} is fulfilled with $c=c_3$ and thus the Proposition follows.
\end{proof}
%For $s\in [0, T_0]~h\in (0, h_0)$ we call  $Y(s,\cdot ; h)$ the inverse of the diffeomorphism $x\mapsto X(s,\cdot ; h)$. To go back from $w_h$ to $u_j$ ones need to make an inverse change of variables to the one we performed above. For this purpose, we prove:
Let $\varphi_1$ be a smooth function verifying 
\[
\supp \varphi_1\subset \{ (2c_2)^{-1}\le |\xi|\le 2c_2\},\quad \varphi_1\equiv 1~\text{in}~\{ (2c_1)^{-1}\le |\xi|\le 2c_1\}.
\]
\begin{lemm} For $f(t, z)=g(t, X(t, \th z))$ we have
\bq\label{change:phi}
\lp\varphi_1(hD_x)g\rp (t, X(t, \th z))=\varphi_h^*(z, D_z)f(t, z)-i\Op(r_h^2)f(t, z),
\eq
with 
\begin{align}
& \varphi^*_h(z, \xi)= \varphi_1\big( M_0(\th z)\th \xi\big),\label{phi*}\\
&r^2_h(z, z', \xi)=\int_0^1\partial_\xi\partial_{z'}\varphi_1\big( M(\th z, \th[z, z']_s)\th \xi\big) J\big( \th z, \th[z, z']_s \big) ds. \label{rh2}
\end{align}
Moreover, for every $(\alpha, \alpha', \xi)\in (\xN^d)^3$ there hold
\begin{align}
&\la \partial_z^\alpha\partial_\xi^\beta  \varphi^*_h(z, \xi)\ra \les_{\alpha, \beta} \th^{-(2\delta-1)|\alpha|+|\beta|}1_{\{\th |\xi|\sim 1 \}},\label{est:phi*}\\
&\la \partial_z^\alpha\partial_{z'}^{\alpha'}\partial_\xi^\beta  r^2_h(z, z',\xi)\ra \les_{\alpha, \alpha',\beta} \th^{2-2\delta(1-r_0)-(2\delta-1)(|\alpha|+|\alpha'|)+|\beta|}1_{\{\th |\xi|\sim 1 \}}.\label{est:rh2}
\end{align}
%Consequently, $r_h^2\in S^{-2+2\delta(1-r_0)}_{1, 2\delta -1, 2\delta-1}$.
\end{lemm}
\begin{proof}
The formulas \eqref{change:phi} and \eqref{phi*}, \eqref{rh2} are derived along the same lines as in paragraphs \ref{Straighten}, \ref{Approximation} where we performed the change of variables $x=X(t, \th z)$ to derive \eqref{eq:wh}.\\
1. Proof of \eqref{est:phi*}.\\
Observe first that 
\[
\partial_\xi^\beta  \varphi^*_h(z, \xi)=(\partial^\gamma \varphi_1)\big( M_0(\th z)\th \xi\big)\big(M_0(\th z)\big)^\gamma \th^{|\beta|}
\]
where $|\gamma|=|\beta|$. Next, Lemma \ref{est:MJ} implies for all $\alpha\in \xN^d$,
\[
\la \partial_z^\alpha (\partial^\gamma \varphi_1)\big( M_0(\th z)\th \xi\big)\ra+\la \partial_z^\alpha \big(M_0(\th z)\big)^\gamma\ra \les \th^{-(2\delta-1)|\alpha|}
\]
and thus \eqref{est:phi*} follows.\\
2. For \eqref{est:rh2} one proceeds exactly as in the proof of Lemma \ref{est:rh}.
\end{proof}
\begin{coro}\label{disp:w1}
If  $g$ is spectrally supported in the annulus $\frac{1}{h}\Cr^1$ then for all $r\in [2, \infty]$ we have
 \begin{multline}
		\lA S_h(t, t_0)\big(g\circ X(t_0,\th z)\big)\rA_{L^r}\les  \th^{-d(\frac{2}{r'}-\mez)}
		   \la t-t_0\ra^{-d(\mez-\frac{1}{r})}\lA g\rA_{L^{r'}}\\
+\lA S(t, t_0)\Op(r_h^2)\big(g\circ X(t_0,\th z)\big)\rA_{L^r}
\end{multline}
	for all~$t,~ t_0\in [0, 1]$ and~$0<\th\leq \th_0$.
\end{coro}
\begin{proof}
We first apply the identity \eqref{change:phi} at $t=t_0$ and notice that $\varphi_1(h\xi)=1$ if $\xi\in \frac{1}{h}\Cr^1$ to have
\[
g\circ X(t_0, \th z)=\varphi_h^*(z, D_z)\big( g\circ X(t_0, \th z)\big)(z)-i\Op(r_h^2)\big( g\circ X(t_0, \th z)\big)(z).
\]
 The estimate \eqref{est:phi*} implies that $\varphi_h^*\in S^0_\ld\cap S^0_{1,\mez}$ ($\ld=\th^{-1})$, so the estimate \eqref{disp:w:inter} applied to $\chi=\varphi_h^*$ and $f(z)=g\circ  X(t_0,\th z)$ gives for all $r\in [2, \infty]$
 \begin{multline*}
		\lA S_h(t, t_0)\big(g\circ X(t_0, \th z)\big)\rA_{L^r}\les \th^{-d(\mez-\frac{1}{r})}
		  \la t-t_0\ra^{-d(\mez-\frac{1}{r})}\lA g\circ X(t_0, \th z)\rA_{L^{r'}}\\
+\lA S(t, t_0)\Op(r_h^2)\big(g\circ X(t_0, \th z)\big)\rA_{L^r}.
\end{multline*}
Finally, since $X$ is Lipschitz we have
\[
\lA g\circ X(t_0, \th z)\rA_{L^{r'}}\les \th^{-\frac{d}{r'}} \lA g\rA_{L^{r'}}
\]
from which we conclude the proof.
\end{proof}
To control the right-hand side of the estimate in the preceding Corollary, we use the following Lemma, whose proof is identical to that of Lemma \ref{energy:0}.
\begin{lemm}\label{energy}
For any $\mu \in \xR$, the operators $S_h(t,s)$, $\mathcal{S}_j(t,s)$ are bounded on $H^\mu(\xR^d)$ uniformly in $t, s\in I$ .
\end{lemm}
Henceforth, we choose $\eps_0>0$ arbitrarily small and
\bq\label{choose:r0r1}
r_1=\frac{1}{6};\quad r_0=\frac{2}{3}+\eps_0 ~\text{when}~d=1,\quad r_0=1~\text{when}~d=2.
\eq
so that 
\begin{align*}
&2\delta(1-r_0)=\mez-\tdm\eps_0,\quad 2\delta(2-r_0)=2-\tdm \eps_0\quad \text{when}~d=1,\\
&2\delta(1-r_0)=0,\quad 2\delta(2-r_0)=\tdm \quad \text{when}~d=2
\end{align*}
and \eqref{r0r1} is fulfilled.\\
The next proposition proves what we called \enquote{pseudo-dispersive estimates} above.
\begin{prop}\label{pseudo:disp}
If $\mathcal{L}_ju_j(t, x)=0$ and $u_j(t),~t\in [0, T]$ are spectrally supported in  $\frac{1}{h}\Cr^1$  then for any $t_0\in [0, T]$ and $r\in [2,\infty]$ we have
\[\lA u_j(t)\rA_{L^r}\les\th^{-d(\frac{2}{r'}-\mez)}
		   \la t-t_0\ra^{-d(\mez-\frac{1}{r})}\lA u_j(t_0)\rA_{L^{r'}}+h^{-\frac{d}{4}} \lA u_j(t_0)\rA_{L^2}
\]
where $r=\infty$ when $d=1$ and $r\in [2,\infty)$ when $d=2$.
\end{prop}
\begin{proof}
First, equation \eqref{eq:whw} implies that $w_h(t,z):=u_j(t, X(t, \th z))$ satisfies 
\[
w_h(t)=S_h(t,t_0)w_h(t_0)-\int_{t_0}^t S(t,s) \lp \Op_h(r_h^1)\rp w_h(s)ds=:(1)+(2).
\]
Applying Corollary \ref{disp:w1} to $g(x)=u_j(t_0, x)$ and then using the Sobolev embeddings
\[
H^{\frac{d}{2}+\eps}\hookrightarrow L^\infty,\quad H^{\frac{d}{2}}\hookrightarrow L^r~\forall r\in [2,\infty)
\]
together with Lemma \ref{energy} one gets
\[
\lA (1)\rA_{L^r}\les \th^{-d(\frac{2}{r'}-\mez)}
		   \la t-t_0\ra^{-d(\mez-\frac{1}{r})}\lA u_j(t_0)\rA_{L^1}+\lA \Op(r_h^2)w_h(t_0)\rA_{H^{\frac{d}{2}+\eps_r}}
\]
where 
\[
\eps_r=0~\text{if}~r\in [2,\infty),\quad \eps_r=\tdm \eps_0~\text{if}~r=\infty.
\]
The estimate \eqref{est:rh2} gives $r_h^2\in S^{-2+2\delta(1-r_0)}_{1,\mez, \mez}$, hence 
\[
\lA \Op(r_h^2)w_h(t_0)\rA_{H^{\frac{d}{2}+\eps_r}}\les \lA w_h(t_0)\rA_{H^{\frac{d}{2}+\eps_r-2+2\delta(1-r_0)}}.%\les \lA u_j\rA_{H^{\frac{d}{2}-1+\eps}}.
\]
Similarly, since $r^1_h\in S^{-1+2\delta(1-r_0)}_{1, \mez, \mez}$ one deduces with the aid of Lemma \ref{energy}
\[
\lA (2)\rA_{L^r}\les \int_{t_0}^t \lA w_h(s)\rA_{H^{\frac{d}{2}+\eps_r-1+2\delta(1-r_0)}}ds.%\les \lA w_h(t_0)\rA_{H^{\frac{d}{2}+\eps_r-1+2\delta(1-r_0)}}.
\]
Putting together the above estimates, one obtains 
\[
\lA w_h(t)\rA_{L^r}\les \th^{-d(\frac{2}{r'}-\mez)}
		   \la t-t_0\ra^{-d(\mez-\frac{1}{r})}\lA u_j(t_0)\rA_{L^{r'}}+ \int_{t_0}^t\lA w_h(s)\rA_{H^{\frac{d}{2}+\eps_r-1+2\delta(1-r_0)}}ds.
\]
When $d=1,~r=\infty$, $\frac{d}{2}+\eps_r-1+2\delta(1-r_0) =0$ and since $X(t,\cdot)\in W^{1,\infty}(\xR^d)$ we have for all $s\in [t_0, t]$
\[ \lA w_h(s)\rA_{H^{\frac{d}{2}+\eps_r-1+2\delta(1-r_0)}}\les  \th^{-\frac{d}{2}}\lA u_j(s)\rA_{L^2}.
\]
When $d=2,~r\in [2, \infty)$, $\frac{d}{2}+\eps_r-1+2\delta(1-r_0)=0$ and thus
\[ 
\lA w_h(s)\rA_{H^{\frac{d}{2}+\eps_r-1+2\delta(1-r_0)}}ds\les  \th^{-\frac{d}{2}}\lA u_j(s)\rA_{L^2}ds\les \th^{-\frac{d}{2}}\lA u_j(s)\rA_{L^2}\quad \forall s\in [t_0, t].
\]
On the other hand, by Lemma \ref{energy}, $\lA u_j(s)\rA_{L^2}\les \lA u_j(t_0)\rA_{L^2}$ for all $s\in [t_0, t]$. Finally, noticing that 
\[
\lA u_j(t)\rA_{L^r}\les \th^{\frac{1}{r}}\lA w_h(t)\rA_{L^r}\les \lA w_h(t)\rA_{L^r}
\]
we conclude the proof.
\end{proof}
\begin{rema}
Strictly speaking, the preceding estimate is not a standard dispersive estimate since it does not show decay in time on the right hand side. The appearance of the non-decaying term $h^{-\frac{d}{4}} \lA u_j(t_0)\rA_{L^2}$ is however harmless for the purpose of proving Strichartz estimates in the next paragraph.
\end{rema}
\subsection{Strichartz estimates}
\begin{prop}\label{Str:uj}
%Let $(q, r)\in [2, \infty],~\frac{2}{q}+\frac{d}{r}=\frac{d}{2}$ with $d=1,~2$ and $(q, r)\ne (2, \infty)$.
Suppose that $\mathcal{L}_ju_j(t, x)=0$ and $u_j(t),~t\in I:=[0, T]$ is spectrally supported in $\frac{1}{h}\Cr^1$. 
\begin{itemize}
\item[(i)] When $d=1$ we have
\bq\label{Str:uj:1d}
\lA u_j\rA_{L^4(I;L^\infty)}\les h^{-\frac{3}{8}}\lA u_j\arrowvert_{t=0}\rA_
{L^2}.
\eq
\item[(ii)] When $d=2$ we have with  $q>2,~\frac{2}{q}+\frac{2}{r}=1$
\bq\label{Str:uj:2d}
\lA u_j\rA_{L^q(I;L^r)}\les h^{\frac{1}{4}-\frac{1}{r'}}\lA u_j\arrowvert_{t=0}\rA_
{L^2}.
\eq
Consequently, for any $s_0\in \xR$ and $\eps>0$
\bq\label{Str:uj:2d:H}
\lA u_j\rA_{L^{2+\eps}(I;W^{s_0-\frac{3}{4}-\eps,\infty})}\les \lA u_j\arrowvert_{t=0}\rA_
{H^{s_0}}.
\eq
\end{itemize}
\end{prop}
\begin{proof}
For the two estimates \eqref{Str:uj:1d}, \eqref{Str:uj:2d}, using the $TT^*$ argument, one need to show that
\[
K:=\int_IS(t,s)ds : L^{q'}(I; L^{r'})\to L^q(I, L^r)
\]
with norm bounded by $h^M$ where $M=-\frac{3}{4}$ when $d=1$ and $M=\frac{1}{2}-\frac{2}{r'}$ when $d=2$. Moreover, since $u_j$ is spectrally supported in $\frac{1}{h}\Cr^1$, it suffices to prove 
\bq\label{est:K}
\lA Kf\rA_{L^q(I, L^r)}\les h^{M} \lA f\rA_{L^{q'}(I; L^{r'})}
\eq
for every $f$ spectrally supported in $\frac{1}{h}\Cr^1$.\\
In view of the "pseudo dispersive estimate" in Proposition \ref{pseudo:disp},
\[
\lA K(t)f\rA_{L^r}\les (1)+(2) , \quad \text{with}
\]
\begin{align*}
& (1)=h^{-d(\frac{1}{r'}-\frac{1}{4})}
		  \int_I
		  \la t-s\ra^{-d(\mez-\frac{1}{r})}\lA f(s)\rA_{L^{r'}}ds,\\
& (2)=h^{-\frac{d}{4}} \int_I\lA f(s)\rA_{L^2}ds.
\end{align*}
$(i)$ $d=1,~(q, r)=(4,\infty)$. By the Hardy-Littlewood-Sobolev inequality, $\lA (1)\rA_{L^q_t}$ is bounded by the right-hand side of \eqref{est:K}. On the other hand, $(2)$ can be estimated using Sobolev embedding as
\[
(2)\les h^{-\frac{d}{4}}h^{-\frac{d}{2}}\lA f(s)\rA_{L^1}ds \les h^{-\frac{3d}{4}} \lA f\rA_{L^1 L^1}
\]
which concludes the proof of \eqref{Str:uj:1d}.\\
$(ii)$ $d=2$, $q>2,~\frac{2}{q}+\frac{2}{r}=1$. Again, the Hardy-Littlewood-Sobolev inequality  yields
\[
\lA (1)\rA_{L^{q}}\les h^{-d(\frac{1}{r'}-\frac{1}{4})}\lA f\rA_{L^{q'}(I; L^{r'})}.
\]
For $(2)$ one uses the embedding $L^{r'}\hookrightarrow L^2,~r'\in [1, 2)$ 
\[
\lA f(s)\rA_{L^2}\les h^{\frac{d}{2}-\frac{d}{r'}}\lA f(s)\rA_{L^{r'}}%=h^{1-\frac{2}{r'}}\lA f(s)\rA_{L^{r'}},
\]
to get 
\[
(2)\les h^{\frac{d}{4}-\frac{d}{r'}}\lA f\rA_{L^1L^{r'}}.
\]
The estimate \eqref{Str:uj:2d} then follows.\\
 Now, for any $\eps>0$, let $q=2+\eps$ then 
\[
\frac{2}{r}=\frac{\eps}{\eps+2},\quad 
\frac{1}{4}-\frac{1}{r'}=-\frac{3}{4}+\frac{\eps}{2(\eps+2)}.
\]
 Multiplying both sides of \eqref{Str:uj:2d} by $h^{-s_0}$ yields
\bq\label{proof:2d}
\lA u_j\rA_{L^{q}(I;W^{{s_0}-\frac{3}{4}+\frac{\eps}{2(\eps+2)}, r})}\les \lA u_j\arrowvert_{t=0}\rA_
{H^{s_0}}.
\eq
 Writing ${s_0}-\frac{3}{4}+\frac{\eps}{2(\eps+2)}=a+b,~a={s_0}-\frac{3}{4}-\eps,~b=\eps+\frac{\eps}{2(\eps+2)}>\frac{d}{r}$ we obtain \eqref{Str:uj:2d:H} from \eqref{proof:2d} and the Sobolev embedding $W^{a+b,r}(\xR^d)\hookrightarrow W^{a, \infty}$ with $b>\frac{d}{r}$.
\end{proof}
Recall from \eqref{defi:r0r1} and \eqref{choose:r0r1} that we have required
\bq\label{condition:V,gamma}
V\in L^\infty(I; W^{\rho,\infty}(\xR^d)),~\gamma(\cdot, \xi)\in L^\infty(I; W^{\frac{2}{3},\infty}(\xR^d))
\eq
where $\rho=\frac{5}{3}+\eps_0,~\eps_0>0$ when $d=1$ and $\rho=2$ when $d=2$.
\begin{theo}
Let ${s_0}\in \xR$, $I=[0, T],~T\in (0, +\infty)$ and $u\in L^\infty(I, H^{s_0}(\xR^d))$ be a solution to the problem 
\[
\partial_tu+ T_V\cdot\nabla  u+iT_\gamma u=f,\quad u\arrowvert_{t=0}=u^0
\]
where $\gamma=\sqrt{\ld a}$ as defined in Theorem \ref{theo:sym}.  \\
1. When $d=1$, if for some $\eps_0>0$, $V\in L^\infty(I, W^{\frac{5}{3}+\eps_0,\infty}(\xR^d)),~\eta\in L^\infty(I, W^{\frac{5}{3},\infty}(\xR^d))$ then 
\[
\lA u\rA_{L^4(I;W^{{s_0}-\frac{3}{8},\infty})}\les \lA u^0\rA_{H^{s_0}}+\lA f\rA_{L^1(I, H^{s_0})}.
\]
2. When $d=2$, if $V\in L^\infty(I, W^{2,\infty}(\xR^d)),~\eta\in L^\infty(I, W^{\frac{5}{3},\infty}(\xR^d))$ then for every $\eps>0$
\[
\lA u\rA_{L^{2+\eps}(I;W^{{s_0}-\frac{3}{4}-\eps,\infty})}\les_\eps \lA u^0\rA_{H^{s_0}}+\lA f\rA_{L^1(I, H^{s_0})}.
\]
In the above estimates, the dependent constants depend on a finite number of semi-norms of the symbols $V$ and $\gamma$.
\end{theo}
\begin{proof}
If $u$ is a solution to \eqref{ww:reduce} with data $u^0$ then by \eqref{eq:reg}, the dyadic piece $\Delta_j u$ is a solution to $\mathcal{L}_j\Delta_j u=F_j$ with $F_j$ given by \eqref{Fj} spectrally supported in $\frac{1}{h}\Cr^1$ with $c_1$ sufficiently large. Under the regularity assumptions of $V,~\gamma$ in $1.$ and $2.$, condition \eqref{condition:V,gamma} is fulfilled. Using Duhamel's formula and applying the Strichartz estimates in Proposition \ref{Str:uj} we deduce that
\bq\label{Str:dyadic}
\lA \Delta_ju\rA_{L^q(I;W^{{s_0}-\frac{d}{2}+\mu,\infty})}\les \lA \Delta_ju^0\rA_{ H^{s_0}}+\lA F_j\rA_{L^1(I; H^{s_0})}, 
\eq
where 
\[
 q=4,~\mu=\frac{1}{8}~\text{when}~d=1;\quad q=2+\eps,~\mu=\frac{1}{4}-\eps~\text{when}~d=2.
\]
We are left with the estimate for $F_j=F_j^1+R_j+F_j^3$ where $F_j^k$ are given by \eqref{Fj1},~\eqref{Fj}. Defining
\[
 \widetilde\Delta_j=\sum_{|k-j|\le 3} \Delta_k,
\]
it follows from \eqref{Fj2} that 
\[
\lA R_j\rA_{H^{s_0}}\les \Vert \widetilde\Delta_ju\Vert_{H^{s_0}}.
\]
Using the symbolic calculus Theorem \ref{theo:sc} one obtains without any difficulty that
\[
\lA F_j^1\rA_{L^1(I; H^{s_0})}\les \Vert \widetilde\Delta_ju\Vert_{L^1(I; H^{s_0})}.
\]
For  $F_j^3$ we use \eqref{Sja:a} to obtain that: if
\[
V\in L^\infty(I, W^{\frac{4}{3},\infty}(\xR^d)),~\eta\in L^\infty(I, W^{\frac{5}{3},\infty}(\xR^d))
\]
then 
\bq\label{est:F3}
\lA F_j^3\rA_{L^1(I; H^{s_0})}\les\lA \Delta_ju\rA_{L^1(I; H^{s_0})}+\lA \Delta_j f\rA_{L^1(I; H^{s_0})}.
\eq
Finally, putting together the above estimates we conclude from \eqref{Str:dyadic} that
\begin{align*}
\lA u\rA_{L^q(I;W^{{s_0}-\frac{d}{2}+\mu,\infty})}&\le \sum_j\lA \Delta_ju\rA_{L^q(I;W^{{s_0}-\frac{d}{2}+\mu,\infty})}\\
&\les \lA u^0\rA_{ H^{s_0}}+\lA u\rA_{L^1(I; H^{s_0})}+\lA f\rA_{L^1(I; H^{s_0})} \\
&\les \lA u^0\rA_{ H^{s_0}}+\lA f\rA_{L^1(I; H^{s_0})},
\end{align*}
where in the last inequality, we have used the energy estimate 
\[
\lA u\rA_{L^\infty(I; H^{s_0})}\les  \lA u^0\rA_{ H^{s_0}}+\lA f\rA_{L^1(I; H^{s_0})}.
\]
The proof is complete.
\end{proof}
According to  Remark \ref{inverse}, after having the estimate for $u$-solution to \eqref{ww:reduce} one can use the symbolic calculus  to obtain the desired estimates for the original solution $(\eta, \psi, B, V)$ as stated in Theorem \ref{main:theo}.\\ \\
%%%%%%%%%%%%%%%%%%%%%%%%%%%%%%%%%%%%%
{\bf Acknowledgment.}~
This work was partially supported by the labex LMH through the grant no ANR-11-LABX-0056-LMH in the "Programme des Investissements d'Avenir". I sincerely thank  Prof. Nicolas Burq and Prof. Claude Zuily for many helpful discussions and encouragements. I thank Prof. Daniel Tataru for suggesting me  his joint work \cite{KT}.
\appendix
\section{}\label{appendix}
\begin{defi}\label{functionalspaces}
1. (Littlewood-Paley decomposition) Let~$\psi \in C^\infty_0({\mathbf{R}}^d)$ be such that
\bq\label{defi:psi}
\psi(\theta)=1\quad \text{for }\la \theta\ra\le 1,\qquad 
\psi(\theta)=0\quad \text{for }\la\theta\ra>2.
\eq
Define
\begin{equation*}
\psi_k(\theta)=\kappa(2^{-k}\theta)\quad\text{for }k\in \xZ,
\qquad \varphi_0=\kappa_0,\quad\text{ and } 
\quad \varphi_k=\psi_k-\psi_{k-1} \quad\text{for }k\ge 1.
\end{equation*}
Given a temperate distribution $u$ and an integer $k$, we introduce 
$S_k u=\psi_k(D_x)u$, $\Delta_k u=S_k u-S_{k-1}u$ for $k\ge 1$ and $\Delta_0u=S_0u$. Then we have the formal decomposition 
\bq\label{partition}
u=\sum_{k=0}^{\infty}\Delta_k u.
\eq
%2. (Zygmund spaces) If~$s$ is any real number, we define the Zygmund class~$C^{s}_*({\mathbf{R}}^d)$ as the 
%space of tempered distributions~$u$ such that
%$$
%\lA u\rA_{C^{s}_*}\defn \sup_q 2^{qs}\lA \Delta_q u\rA_{L^\infty}<+\infty.
%$$
2. (H\"older spaces) For~$k\in\xN$, we denote by $W^{k,\infty}({\mathbf{R}}^d)$ the usual Sobolev spaces.
For $\rho= k + \sigma$, $k\in \xN, \sigma \in (0,1)$ denote 
by~$W^{\rho,\infty}({\mathbf{R}}^d)$ 
the space of functions whose derivatives up to order~$k$ are bounded and uniformly H\"older continuous with 
exponent~$\sigma$. 
\end{defi}
\begin{defi}\label{spectral}
Let  $u$ be a tempered distribution in $\xR^d$. We define the spectrum of $u$ to be the support of the Fourier transform of $ u$. Then $u$ is said to be spectrally supported in a set $A\subset \xR^d$ if the spectrum of $u$ is contained in $A$.
\end{defi}
Let us review notations and results about Bony's paradifferential calculus (cf. \cite{Bony}, \cite{MePise}).
\begin{defi}\label{defi:para}
1. (Symbols) Given~$\rho\in [0, \infty)$ and~$m\in\xR$,~$\Gamma_{\rho}^{m}({\mathbf{R}}^d)$ denotes the space of
locally bounded functions~$a(x,\xi)$
on~${\mathbf{R}}^d\times({\mathbf{R}}^d\setminus 0)$,
which are~$C^\infty$ with respect to~$\xi$ for~$\xi\neq 0$ and
such that, for all~$\alpha\in\xN^d$ and all~$\xi\neq 0$, the function
$x\mapsto \partial_\xi^\alpha a(x,\xi)$ belongs to~$W^{\rho,\infty}({\mathbf{R}}^d)$ and there exists a constant
$C_\alpha$ such that,
\bq\label{para:symbol}
\forall\la \xi\ra\ge \mez,\quad 
\lA \partial_\xi^\alpha a(\cdot,\xi)\rA_{W^{\rho,\infty}(\xR^d)}\le C_\alpha
(1+\la\xi\ra)^{m-\la\alpha\ra}.
\eq
Let $a\in \Gamma_{\rho}^{m}({\mathbf{R}}^d)$, we define the semi-norm
\begin{equation}\label{defi:semi-norm}
M_{\rho}^{m}(a)= 
\sup_{\la\alpha\ra\le 2(d+2) +\rho ~}\sup_{\la\xi\ra \ge 1/2~}
\lA (1+\la\xi\ra)^{\la\alpha\ra-m}\partial_\xi^\alpha a(\cdot,\xi)\rA_{W^{\rho,\infty}({\mathbf{R}}^d)}.
\end{equation}
2. (Paradifferential operators) Given a symbol~$a$, we define
the paradifferential operator~$T_a$ by
\begin{equation}\label{eq.para}
\widehat{T_a u}(\xi)=(2\pi)^{-d}\int \chi(\xi-\eta,\eta)\widehat{a}(\xi-\eta,\eta)\rho(\eta)\widehat{u}(\eta)
\, d\eta,
\end{equation}
where
$\widehat{a}(\theta,\xi)=\int e^{-ix\cdot\theta}a(x,\xi)\, dx$
is the Fourier transform of~$a$ with respect to the first variable; 
$\chi$ and~$\rho$ are two fixed~$C^\infty$ functions such that:
\begin{equation}\label{cond.psi}
\rho(\eta)=0\quad \text{for } \la\eta\ra\le \frac{1}{5},\qquad
\rho(\eta)=1\quad \text{for }\la\eta\ra\geq \frac{1}{4},
\end{equation}
and~$\chi(\theta,\eta)$  is defined by
$
\chi(\theta,\eta)=\sum_{k=0}^{+\infty} \kappa_{k-3}(\theta) \varphi_k(\eta).
$
\end{defi}
We remark that the cut-off function $\chi$ in the preceding definition has the following properties for some $0<\eps_1<\eps_2<1$
\bq\label{chi:prop}
\begin{cases}
\chi(\eta, \xi)=1& \text{for}~|\eta|\le \eps_1(1+|\xi),\\
\chi(\eta, \xi)=0&\text{for}~|\eta|\ge \eps_2(1+|\xi).
\end{cases}
\eq
\begin{defi}\label{defi:order}
Let~$m\in\xR$.
An operator~$T$ is said to be of  order~$\leo m$ if, for all~$\mu\in\xR$,
it is bounded from~$H^{\mu}$ to~$H^{\mu-m}$. 
\end{defi}
Symbolic calculus for paradifferential operators is summarized in the following theorem.
\begin{theo}\label{theo:sc}(Symbolic calculus)
Let~$m\in\xR$ and~$\rho\in [0, \infty)$. \\
$(i)$ If~$a \in \Gamma^m_0({\mathbf{R}}^d)$, then~$T_a$ is of order~$\leo m$. 
Moreover, for all~$\mu\in\xR$ there exists a constant~$K$ such that
\begin{equation}\label{esti:quant1}\lA T_a \rA_{H^{\mu}\rightarrow H^{\mu-m}}\le K M_{0}^{m}(a).
\end{equation}
$(ii)$ If~$a\in \Gamma^{m}_{\rho}({\mathbf{R}}^d), b\in \Gamma^{m'}_{\rho}({\mathbf{R}}^d)$ then 
$T_a T_b -T_{a\sharp  b}$ is of order~$\leo m+m'-\rho$ with
\[
a\sharp b:=\sum_{|\alpha|<\rho}\frac{(-i)^{\alpha}}{\alpha !}\partial_{\xi}^{\alpha}a(x, \xi)\partial_x^{\alpha}b(x, \xi).
\] 
Moreover, for all~$\mu\in\xR$ there exists a constant~$K$ such that
\begin{equation}\label{esti:quant2}
\lA T_a T_b  - T_{a \sharp b}   \rA_{H^{\mu}\rightarrow H^{\mu-m-m'+\rho}}%&
\le 
K M_{\rho}^{m}(a)M_{0}^{m'}(b)+K M_{0}^{m}(a)M_{\rho}^{m'}(b).
\end{equation}
$(iii)$ Let~$a\in \Gamma^{m}_{\rho}({\mathbf{R}}^d)$. Denote by 
$(T_a)^*$ the adjoint operator of~$T_a$ and by~$\overline{a}$ the complex conjugate of~$a$. Then $(T_a)^* -T_{b}$ is of order~$\leo m-\rho$ with
\[
b:=\sum_{|\alpha|<\rho}\frac{(-i)^{\alpha}}{\alpha !}\partial_{\xi}^{\alpha}\partial_x^\alpha \bar a(x, \xi).
\]
Moreover, for all~$\mu$ there exists a constant~$K$ such that
\begin{equation}\label{esti:quant3}
\lA (T_a)^*   - T_b   \rA_{H^{\mu}\rightarrow H^{\mu-m+\rho}}\le 
K M_{\rho}^{m}(a).
\end{equation}
\end{theo}


\begin{thebibliography}{10}
\small
%\bibitem{ABZ}
%Thomas Alazard, Nicolas Burq, and Claude Zuily.
%\newblock On the Cauchy problem for gravity water waves.
%\newblock {\em Invent.Math.}, 198(1): 71--163, 2014.

%\bibitem{ABZ'}
%Thomas Alazard, Nicolas Burq and Claude Zuily. 
%\newblock Cauchy theory for the gravity water waves system with nonlocalized initial data.
%\newblock  arXiv:1305.0457.

\bibitem{ABZ1}
Thomas Alazard, Nicolas Burq, and Claude Zuily.
\newblock On the water waves equations with surface tension.
\newblock {\em Duke Math. J.}, 158(3):413--499, 2011.

\bibitem{ABZ2}
Thomas Alazard, Nicolas Burq, and Claude Zuily.
\newblock Strichartz estimates for water waves.
\newblock {\em Ann. Sci. {\'E}c. Norm. Sup{\'e}r. (4)}, 44(5):855--903, 2011.


\bibitem{ABZ3}
Thomas Alazard, Nicolas Burq, and Claude Zuily.
\newblock On the Cauchy problem for gravity water waves.
\newblock {\em Invent.Math.}, 198(1): 71--163, 2014. 

\bibitem{ABZ4}
Thomas Alazard, Nicolas Burq, and Claude Zuily.
\newblock Strichartz estimate and the Cauchy problem for the gravity water waves equations.
\newblock{ \em arXiv:1404.4276}, 2014.

\bibitem{AD}
Thomas Alazard, Jean-Marc Delort.
\newblock Global solutions and asymptotic behavior for two dimensional gravity water waves.
\newblock{\em Ann. Sci. \'Ec. Norm. Sup\'er.}, to appear. 100 pages.

\bibitem{Alipara}
Serge Alinhac.
\newblock Paracomposition et op\'erateurs paradiff\'erentiels.
\newblock {\em Comm. Partial Differential Equations}, 11(1):87--121, 1986.
%\bibitem{AM}
%Thomas Alazard and Guy M{\'e}tivier.
%\newblock Paralinearization of the {D}irichlet to {N}eumann operator, and
  %regularity of three-dimensional water waves.
%\newblock {\em Comm. Partial Differential Equations}, 34(10-12):1632--1704,
 % 2009.


%\bibitem{BCD}
%Hajer Bahouri, Jean-Yves Chemin, and Rapha{\"e}l Danchin.
%\newblock {\em Fourier analysis and nonlinear partial differential equations},
  %volume 343 of {\em Grundlehren der Mathematischen Wissenschaften [Fundamental
 % Principles of Mathematical Sciences]}.
%\newblock Springer, Heidelberg, 2011.

 
\bibitem{Bony}
Jean-Michel Bony.
\newblock Calcul symbolique et propagation des singularit\'es pour les
  \'equations aux d\'eriv\'ees partielles non lin\'eaires.
\newblock {\em Ann. Sci. \'Ecole Norm. Sup. (4)}, 14(2):209--246, 1981.

\bibitem{BGT}
Nicolas Burq, Patrick G\'erard, and Nikolay Tzvetkov.
\newblock Strichartz inequalities and the nonlinear {S}chr\"odinger equation on
  compact manifolds.
\newblock {\em Amer. J. Math.}, 126(3):569--605, 2004.

\bibitem{CHS}
Hans Christianson, Vera~Mikyoung Hur, and Gigliola Staffilani.
\newblock Strichartz estimates for the water-wave problem with surface tension.
\newblock {\em Comm. Partial Differential Equations}, 35(12):2195--2252, 2010.

\bibitem{NgPo1}
Thibault de Poyferr\'{e} and Quang-Huy Nguyen.
\newblock A paradifferential reduction for the gravity-capillary waves system at low regularity and applications.
\newblock{\em  arXiv:1508.00326}, 2015.

\bibitem{NgPo2}
Thibault de Poyferre and Quang-Huy Nguyen.
\newblock Strichartz estimates and local existence for the capillary water waves with non-Lipschitz initial velocity.
\newblock{\em Journal of Differential Equations}, 261(1):396--438, 2016, 

\bibitem{GMS}
Pierre Germain, Nader Masmoudi, and Jalal Shatah.
\newblock Global solutions for the gravity water waves equation in dimension 3.
\newblock {\em Annals of Mathematics}, 175(2):691--754, 2012.
 
\bibitem{GMS1}
Pierre Germain, Nader Masmoudi, and Jalal Shatah.
\newblock Global existence for capillary water waves.
\newblock{\em Comm. Pure Appl. Math.}, 68(4): 625--687, 2015.

\bibitem{HuIfTa}
John Hunter, Mihaela Ifrim, Daniel Tataru.
\newblock Two dimensional water waves in holomorphic coordinates.
\newblock{\em  arXiv:1401.1252v2}, 2014.

\bibitem{IfTa}
Mihaela Ifrim, Daniel Tataru.
\newblock Two dimensional water waves in holomorphic coordinates II: global solutions.
\newblock{\em arXiv:1404.7583}, 2014.

\bibitem{IfTa2}
Mihaela Ifrim, Daniel Tataru.
\newblock 
The lifespan of small data solutions in two dimensional capillary water waves.
\newblock{\em  arXiv:1406.5471}, 2014.

\bibitem{IoPu}
Alexandru D. Ionescu, Fabio Pusateri.
\newblock Global solutions for the gravity water waves system in 2d.
\newblock{\em Inventiones mathematicae},  Volume 199, Issue 3, pp 653--804, 2015.

\bibitem{IoPu1}
Alexandru D. Ionescu, Fabio Pusateri.
\newblock Global regularity for 2d water waves with surface tension .
\newblock{\em arXiv:1408.4428},  2015.

%\bibitem{Hormander}
%Lars H{\"o}rmander.
%\newblock {\em Lectures on nonlinear hyperbolic differential equations},
 % volume~26 of {\em Math\'ematiques \& Applications (Berlin) [Mathematics \&
 % Applications]}.
%\newblock Springer-Verlag, Berlin, 1997.

\bibitem{KT}
  Herbert Koch and Daniel Tataru.
\newblock Dispersive estimates for principally normal pseudo-differential operators
\newblock{\em Comm. Pure Appl. Math.}, 58 (2005), no. 2, 217–284.

\bibitem{LannesLivre}
David Lannes.
\newblock {\em Water waves: mathematical analysis and asymptotics.}
\newblock Mathematical Surveys and Monographs, 188. American Mathematical Society, Providence, RI, 2013. 

%\bibitem{LannesJAMS}
%David Lannes.
%\newblock Well-posedness of the water waves equations.
%\newblock {\em J. Amer. Math. Soc.}, 18(3):605--654 (electronic), 2005.


\bibitem{LindbladAnnals}
Hans Lindblad.
\newblock Well-posedness for the motion of an incompressible liquid with free
  surface boundary.
\newblock {\em Ann. of Math. (2)}, 162(1):109--194, 2005.

\bibitem{MePise}
Guy M{\'e}tivier.
\newblock {\em paradifferential calculus and applications to the {C}auchy
  problem for nonlinear systems}, volume~5 of {\em Centro di Ricerca Matematica
  Ennio De Giorgi (CRM) Series}.
\newblock Edizioni della Normale, Pisa, 2008.

\bibitem{Ng}
Quang-Huy Nguyen.
\newblock A sharp Cauchy theory for 2D gravity-capillary water waves.
\newblock {\em arXiv:1601.07442}, 2016.
\bibitem{StTa}
Gigliola Staffilani and Daniel Tataru.
\newblock Strichartz estimates for a {S}chr\"odinger operator with nonsmooth
  coefficients.
\newblock {\em Comm. Partial Differential Equations}, 27(7-8):1337--1372, 2002.

%\bibitem{Tataru:FP}
%Daniel Tataru.
%\newblock On the Fefferman-Phong inequality and related problems.
%\newblock {\em  Comm. Partial Differential Equations}, 27(11-12): 2101--2138, 2007.

\bibitem{Taylor}
Michael~E. Taylor.
\newblock {\em Pseudo-differential operators and nonlinear {PDE}}, volume 100 of
  {\em Progress in Mathematics}.
\newblock Birkh\"auser Boston, Inc., Boston, MA, 1991.

\bibitem{Wu09}
Sijue Wu.
\newblock Almost global wellposedness of the 2-{D} full water waves problem.
\newblock {\em Invent. Math.}, 177(1):45--135, 2009.

\bibitem{Wu11}
Sijue Wu.
\newblock Global wellposedness of the 3-{D} full water waves problem.
\newblock {\em Invent. Math.}, 184(1):125--220, 2011.

\bibitem{Zworski}
Maciej Zworski.
\newblock {\em Semiclassical analysis}, volume 138 of {\em Graduate Studies in
  Mathematics}.
\newblock American Mathematical Society, Providence, RI, 2012.
\end{thebibliography}
\end{document}